\documentclass[12pt]{amsart}

\title[Global Bifurcation of Interfacial Waves]{Global Bifurcation Theory for Periodic Traveling Interfacial Gravity-Capillary Waves}
\author{David M. Ambrose} 
\address{Department of Mathematics, Drexel University, Philadelphia, PA 19104}
\thanks{DMA gratefully acknowledges support from the National Science Foundation through
grant DMS-1016267.}
\author{Walter A. Strauss} 
\address{Department of Mathematics, Brown University, Providence, RI 02912}
\thanks{WAS gratefully acknowledges support from the National Science Foundation through grant
DMS-1007960.}
\author{J. Douglas Wright}
\address{Department of Mathematics, Drexel University, Philadelphia, PA 19104}
\thanks{JDW acknowledges gratefully  support from the National Science Foundation through
grant DMS-1105635.}

\newtheorem{theorem}{Theorem}
\newtheorem{proposition}{Proposition}
\newtheorem{cor}{Corollary}
\newtheorem{lemma}[theorem]{Lemma}
\newtheorem{definition}{Definition}
\theoremstyle{remark}
\newtheorem{remark}{Remark}

\usepackage{amsmath}
\usepackage{amssymb} 
\usepackage{enumerate}

\newcommand{\ep}{\epsilon}

\newcommand{\sgn}{{\textrm{ sgn}}}

\newcommand{\be}{\begin{equation}}
\newcommand{\ee}{\end{equation}}
\newcommand{\bes}{\begin{equation*}}
\newcommand{\ees}{\end{equation*}}
\newcommand{\R}{{\bf{R}}}
\newcommand{\C}{{\bf{C}}}

\newcommand{\Z}{{\bf{Z}}}
\newcommand{\N}{{\bf{N}}}
\newcommand{\ds}{\displaystyle}
\newcommand{\LV}{\left\vert}
\newcommand{\RV}{\right\vert}

\renewcommand{\H}{{\mathcal{H}}}

\newcommand{\U}{{\mathcal{U}}}

\newcommand{\CA}{{\mathcal{C}}}
\newcommand{\per}{{\text{per}}}
\newcommand{\thetab}{\breve{\theta}}
\newcommand{\gammab}{\breve{\gamma}}

\newcommand{\paa }{\partial_\alpha^{-2} }

\begin{document}

\begin{abstract}
We consider the global bifurcation problem for spatially periodic traveling waves for two-dimensional gravity-capillary vortex sheets. 
The two fluids have arbitrary constant, non-negative densities (not both zero),
the gravity parameter can be positive, negative, or zero, and the surface tension
parameter is positive.  Thus, included in the parameter set are
the cases of pure capillary water waves and gravity-capillary water waves.
Our choice of coordinates  allows for the possibility that the fluid interface is not a graph over the horizontal. 
We use a technical reformulation 
which converts the traveling wave equations into a system of the form ``identity plus compact." Rabinowitz' 
global bifurcation theorem
is applied and the final conclusion is the existence of either a closed loop of solutions, or
an unbounded set of  nontrivial traveling wave solutions 
which contains waves which may move arbitrarily fast, become arbitrarily long,
form  singularities in 
the vorticity or curvature, or whose interfaces self-intersect.
\end{abstract}
 
\maketitle

\section{Introduction}

We consider the case of two two-dimensional 
fluids, of infinite vertical extent and periodic in the horizontal direction (of period $M > 0$) and separated by an interface which is free to move.  Each fluid has a constant, non-negative density: $\rho_2\ge 0$ in the upper fluid and $\rho_1\ge 0$ in the lower.
Of course, we do not allow both densities to be zero, but if one of the densities is zero, then it is known as the water wave case.  The velocity of each fluid satisfies the
incompressible, irrotational Euler equations.  The restoring forces in the problem include  non-zero surface tension (with surface tension constant $\tau > 0$)
on the interface
and a  gravitational body force (with acceleration $g \in \R$, possibly zero)
which acts in the vertical direction.
Since the fluids are irrotational, the interface is
a vortex sheet, meaning that the vorticity in the problem is an amplitude times a Dirac mass  supported on the interface.
We call this problem ``the two-dimensional gravity-capillary vortex sheet problem."  
The average vortex strength on the interface is denoted by $\overline\gamma$.  

In \cite{AAW}, two of the authors and Akers established a new formulation for the traveling wave problem for parameterized
curves, and applied it to the vortex sheet with surface tension (in case the two fluids have the same density).  
The curves in \cite{AAW} may have multi-valued height.  This is significant
since it is known that there exist traveling waves in the presence of surface tension
which do indeed have multi-valued height; the most famous such waves are the Crapper
waves \cite{Crapper}, and there are other, related waves known \cite{kinnersley},
\cite{AAW2}, \cite{deBoeck2}.
The results of \cite{AAW} were both analytical and 
computational; the analytical conclusion was a local bifurcation theorem, demonstrating that there exist traveling vortex sheets with surface tension
nearby to equilibrium.  In the present work, we establish a {\it global} bifurcation theorem for the problem with general densities.  We now state a somewhat informal version of this theorem:

\begin{theorem} \label{main result}  
{\bf (Main Theorem)} 
For all choices of the constants $\tau > 0$, $M > 0$, $\overline\gamma\in\R$, 
$\rho_1,\rho_2 \ge 0$ (not both zero) and  $g \in \R$,  
there exist a countable number of connected sets of smooth\footnote 
{Here and below, when we say a function is ``smooth" we mean that its derivatives of all orders exist.} 
non-trivial symmetric periodic traveling wave solutions, bifurcating from a quiescent equilibrium,   
for  the two-dimensional gravity-capillary vortex sheet problem.
If $\bar{\gamma}\neq 0$ or $\rho_{1}\neq\rho_{2},$ then each of these
connected sets has at least one of the following properties:
\begin{enumerate}[(a)]
\item it contains waves whose interfaces have lengths per period which are arbitrarily long;
\item it contains waves whose interfaces have arbitrarily large curvature; 
\item it contains waves where the jump of the tangential component  of the fluid velocity  
 across the interface 
 or its derivative is arbitrarily large;
\item its closure contains a wave whose interface has a point of self intersection;
\item it contains a sequence of waves whose interfaces converge to a flat configuration but whose speeds contain at least two convergent subsequences whose limits differ.
\end{enumerate}
In the case that $\bar{\gamma}=0$ and $\rho_{1}=\rho_{2},$ then each connected
set has at least one of the properties (a)-(f), where (f) is the following:
\begin{enumerate}[(f)]
\item it contains waves which have speeds which are arbitrarily large.
\end{enumerate}

\end{theorem}

We mention that in the case of pure gravity waves, it has sometimes been possible to 
rule out the possibility of an outcome like (e) above; one such paper, for example, is
\cite{constantinStrauss}.  The argument to eliminate such an outcome is typically
a maximum principle argument, and this type of argument appears to be unavailable in the
present setting because of the larger number of derivatives stemming from the presence of
surface tension.  In a forthcoming numerical work, computations will be presented which indicate
that in some cases, outcome (e) can in fact occur for gravity-capillary waves \cite{aapwPreprint}.

Following \cite{AAW}, we start from the formulation of the problem introduced by Hou, Lowengrub, and Shelley, which uses geometric dependent variables and a 
normalized arclength parameterization of the free surface \cite{HLS1}, \cite{HLS2}.  
This formulation follows from the observation that the tangential velocity can be
chosen arbitrarily, while only the normal velocity needs to be chosen in accordance with the physics of the problem.  The tangential velocity can then
be selected in a convenient fashion which  allows us to
 specialize the equations of motion to the periodic traveling wave case in a way that  does not require the interface to be a graph over the horizontal coordinate.
The resulting equations are nonlocal, nonlinear and involve the singular Birkhoff-Rott integral. Despite their complicated appearance, using several well-known properties of the Birkhoff-Rott integral
we are able to recast the traveling wave equations in the form of ``identity plus compact." Consequently, we are able to use an abstract version of the Rabinowitz global-bifurcation theory \cite{R} to prove our main result.  An interesting feature of our formulation is that, unlike similar formulations that allow for overturning waves
by using a conformal mapping, an extension of the present method to the case of 3D waves, 
using for instance ideas like those in \cite{ambroseMasmoudi2},  
seems entirely possible.

The main theorem allows for both positive and negative gravity;  
equivalently, we could say we allow a heavier fluid above or below a lighter fluid.
As remarked in \cite{AAW2}, this is an effect that relies strongly on the presence
of surface tension.  In the case of pure gravity waves, there are some
theorems in the literature demonstrating the nonexistence of traveling waves in
the case of negative gravity \cite{hur-nonlinearity}, \cite{toland-pseudo}.

A similar problem was treated by Amick and Turner \cite{amickTurner1}.  
As with the present paper they treat the global bifurcation of interfacial
waves between two fluids.  However, they require the non-stagnation
condition that the horizontal velocity of the fluid is less than the
wave speed ($u<c$).  Thus their global connected set stops once $u=c$ and
there cannot be any overturning waves.  Their paper has some other
less important differences as well, namely it treats solitary waves and the
top and bottom are fixed ($0<y<1$).  Their methodology is very different
from ours as well, since they handle the case of a smooth density first 
without using the Birkhoff-Rott formulation, and only later let the density
approach a step function.  Another paper \cite{amickTurner2} by the same authors
only treats small solutions. 

Global bifurcation with $\rho_2\equiv 0$, that is, in the water wave case,
has been studied by a variety of authors.  In particular, global
bifurcation that permits overturning waves in the case of constant
vorticity is treated in \cite{walterPreprint}.  Another recent paper is \cite{deBoeck},
in which a global bifurcation theorem is proved in the case $\rho_{2}\equiv0$ for capillary-gravity
waves on finite depth, also with constant vorticity.
Both of these works allow for multi-valued waves by means 
of a conformal mapping.
Walsh treats global bifurcation for capillary water waves with general non-constant vorticity in \cite{walsh},
with the requirement that the interface be a graph with respect to the horizontal coordinate.
The methodologies of all of these papers are completely different from the present work.

Our reformulation of the traveling wave problem into the form ``identity plus compact'' uses the presence
of surface tension in a fundamental way.  In particular, the surface tension enters the problem through the
curvature of the interface, and the curvature involves derivatives of the free surface.  By inverting these derivatives,
we gain the requisite compactness.  The paper \cite{matioc} uses a similar idea to gain compactness 
in order to prove
a global bifurcation theorem for capillary-gravity water waves with constant vorticity and single-valued height.

We mention that the current work finds examples of solutions for interfacial irrotational flow which exist for all time.
The relevant initial value problems are known to be well-posed at short times \cite{A}, but behavior at large times
is in general still an open question.  Some works on existence or nonexistence of singularities for these problems
are \cite{splash1}, \cite{splash2}, \cite{splash3}.  For small-amplitude, pure capillary water waves, global solutions
are known to exist in general \cite{globalCapillary}.

The plan of the paper is as follows: in Section 2, we describe the equations of motion for the relevant interfacial 
fluid flows.  In Section 3, we detail our traveling wave formulation which uses the arclength formulation and which
allows for waves with multi-valued height.  In Section 4, we explore the consequences of the assumption of spatial
periodicity for our traveling wave formulation.  In Section 5, we continue to work with the traveling wave formulation,
now reformulating into an equation of the form ``identity plus compact.''  This sets the stage for Section 6, in which
we state a more detailed version of our main theorem and provide the proof.

\section{The Equations of Motion}\label{eom}

If we make the canonical identification\footnote{Throughout this paper we make this identification for any vector in $\R^2$.} of ${\bf R}^2$ with the complex plane ${\bf C}$, we may represent the free surface at time $t$, denoted by $S (t)$, as the graph (with respect to the parameter $\alpha$) of
$$ 
z(\alpha,t) = x(\alpha,t) + i y(\alpha,t).
$$
The unit tangent and upward normal vectors  to $S $ are, respectively:
\begin{equation}\label{tangent def}
T = {z_\alpha\over |z_\alpha|}  \text{ and } N = i {z_\alpha\over |z_\alpha|}.
\end{equation}
(A derivative with respect to $\alpha$ is denoted either as a subscript or as $\partial_\alpha$.)  
Thus we have uniquely defined real valued functions $U(\alpha,t)$ and $V(\alpha,t)$ such that
\begin{equation}\label{z eqn}
z_t  = U N + V T
\end{equation}
for all $\alpha$ and $t$. We call $U$ the normal velocity of the interface and $V$ the tangential velocity.
The normal velocity $U$ is determined from fluid mechanical considerations and is given by:
\begin{equation}\label{this is U}
U= \Re(W^* N)
\end{equation}
where
\begin{equation}\label{B}
W^*(\alpha,t):=\frac{1}{2 \pi i }\mathrm{PV}\int_{\R} {\gamma(\alpha',t) \over z(\alpha,t) -z(\alpha',t)} d \alpha' 
\end{equation}
is commonly referred to as the Birkhoff-Rott integral. (We use ``$*$" to denote complex conjugation.)

The real-valued quantity $\gamma$ is called in \cite{HLS1} ``the unnormalized vortex sheet-strength," though in this document we
will primarily refer to it as simply the ``vortex sheet-strength."
It
can be used to recover the Eulerian fluid velocity (denoted by $u$) in the bulk at time $t$ and position $w \notin S (t)$  via
\be\label{velocity}
u(w,t):=\left[\frac{1}{2 \pi i }\int_{\R} {\gamma(\alpha',t) \over w -z(\alpha',t)} d \alpha' \right]^*.
\ee
The quantity $\gamma$ is also related to the jump in the tangential velocity of the fluid. 
Specifically, using the Plemelj formulas, one finds that:
$$
[[ u ]] := \lim_{ w \to z(\alpha,t)^+} u(w,t) - \lim_{ w \to z(\alpha,t)^-} u(w,t) = { \gamma(\alpha,t)\over z^*_\alpha(\alpha,t)}.
$$
In the above, the ``$+$" and ``$-$" modifying $z(\alpha,t)$ mean that the limit is taken from ``above" or ``below" $S (t)$, respectively.
If we let $j(\alpha,t):=\Re( [[u]]^* T)$ be the component of $[[u]]$ which is tangent to $S (t)$ at $z(\alpha,t)$, then
the preceding formula shows:
\be\label{jump}
\gamma(\alpha,t) = j(\alpha,t)\ |z_\alpha(\alpha,t)|,   
\ee
which is to say that $\gamma(\alpha,t)$ is a scaled version of the {\it jump in the tangential velocity} of the fluid 
across the interface.

As shown in \cite{A}, $\gamma$ evolves according to the equation 
\begin{multline}\label{gamma eqn v1}
\gamma_{t}=\tau\frac{\theta_{\alpha\alpha}}{|z_\alpha|}+\frac{((V-
\Re( W^* T)
)\gamma)_{\alpha}}{|z_\alpha|}\\
-2A\left(\frac{
\Re(W_t^* T)
}{|z_\alpha|}+\frac{1}{8}\frac{(\gamma^{2})_{\alpha}}{|z_\alpha|^{2}}
+gy_{\alpha}-(V-
\Re(W^* T)
)
\Re(W^*_\alpha T)
\right).
\end{multline}
Here $A$ is the {\it Atwood number},
$$A:=\frac{\rho_{1}-\rho_{2}}{\rho_{1}+\rho_{2}}.$$
Note that $A$ can be taken as any value in the interval $[-1,1].$
Lastly, $\theta(\alpha,t)$ is the {\it tangent angle} to $S (t)$ at the point $z(\alpha,t)$. Specifically it is defined by the relation
$$
z_\alpha = |z_\alpha| e^{i \theta}.
$$
Observe that we have the following nice representations of the tangent and normal vectors in terms of $\theta$:
\be\label{nice T}
T={e^{i \theta}} \text{ and } N={ie^{i \theta}}.
\ee

As observed above, the tangential velocity $V$ has no impact on the geometry of $S (t)$. As such, we are free to make $V$ anything we wish.
In this way, one sees that equations \eqref{z eqn} and \eqref{gamma eqn v1} form a closed dynamical system.
In \cite{HLS1}, the authors make use of the flexibility in the choice of $V$ to design an efficient and non-stiff numerical method for the solution of the dynamical system. In the article \cite{A}, 
 $V$ is selected in a way which is helpful in making {\it a priori} energy estimates, and in completing
 a proof of  local-in-time well-posedness of the initial value problem.
We leave $V$ arbitrary for now.

\section{Traveling waves}

We are interested in finding  traveling wave solutions, which is to say  solutions where both the interface and Eulerian fluid velocity 
propagate horizontally with no change in form and at constant speed.
To be precise:
\begin{definition} \label{TW def} We say $(z(\alpha,t),\gamma(\alpha,t))$ is a traveling wave solution of \eqref{z eqn} and \eqref{gamma eqn v1}
if there exists $c \in \R$ such that for all $t \in \R$ we have
\begin{equation}\label{interface translates}
S (t) = S (0) + ct
\end{equation}
and, for all $w \notin S (t)$,
\begin{equation}\label{velocity translates}
u(w,t) = u(w-ct,0)
\end{equation}
where $u$ is determined from $(z(\alpha,t),\gamma(\alpha,t))$ by way of  \eqref{velocity}.
\end{definition}

Later on the speed $c$ will serve as our bifurcation parameter.  
We have the following results concerning traveling wave solutions of \eqref{z eqn}~and~\eqref{gamma eqn v1}.
\begin{proposition}\label{TW eqn lemma}{\bf (Traveling wave ansatz)}
(i) Suppose that $(z(\alpha,t),\gamma(\alpha,t))$ solves \eqref{z eqn} and \eqref{gamma eqn v1}
and, moreover, there exists $c \in \R$ such that
\begin{equation}\label{TW v1}
{z}_t = c \quad \text{and} \quad
{\gamma}_t  =0
\end{equation}
holds for all $\alpha$ and $t$. Then $(z(\alpha,t),\gamma(\alpha,t))$ is a traveling wave solution with speed~$c$.

(ii) If $(\check{z},\check{\gamma})$ is a traveling wave solution with speed~$c$ of \eqref{z eqn} and \eqref{gamma eqn v1}
then there exists a reparameterization of $S (t)$ which maps $(\check{z},\check{\gamma}) \mapsto (z,\gamma)$ where 
$(z,\gamma)$ 
satisfies~\eqref{TW v1}.
\end{proposition}

\begin{proof} 
First we prove (i).
Since $z_t = c$, we have $z(\alpha,t) = z(\alpha,0) + c t$ which immediately gives \eqref{interface translates}.
Then, since $\gamma_t = 0$ we have $\gamma(\alpha,t) = \gamma(\alpha,0)$ and thus
\begin{equation} 
u^*(w,t)=\frac{1}{2 \pi i }\int_{\R} {\gamma(\alpha',t) \over w -z(\alpha',t)} d \alpha'\\=
\frac{1}{2 \pi i }\int_{\R} {\gamma(\alpha',0) \over w -(z(\alpha',0) +ct)} d \alpha' = u^*(w-ct,0).
\end{equation}
And so we have \eqref{velocity translates}.

Now we prove (ii). 
Suppose $(\check{z}(\beta,t),\check{g}(\beta,t))$ gives a traveling wave solution. The reparameterization  which 
yields \eqref{TW v1} can be written explicitly. Specifically, condition  \eqref{interface translates} 
implies that $z(\alpha,t) := \check{z}(\alpha,0) + ct$ is a parameterization of $S (t)$. Clearly $z_t = c$,
and we have the first equation in \eqref{TW v1}. 

Now let $\gamma(\alpha,t)$ be the corresponding vortex sheet-strength for the parameterization of $S (t)$
given by $z(\alpha,t)$. Since we have a traveling wave, we have \eqref{velocity translates}.   Define 
$$   
m(w,t) =:  {1 \over 2 \pi i} \int_\R {\gamma(\alpha',t) -\gamma(\alpha',0)\over w-ct -\check{z}(\alpha',0)} d \alpha'.  $$
Then for $w \notin S (t)$ we have 
\begin{equation}\begin{split}
m(w,t)  & 
= {1 \over 2 \pi i} \int_\R {\gamma(\alpha',t)\over w-(\check{z}(\alpha',0)+ct)} d \alpha - {1 \over 2 \pi i} \int_\R {\gamma(\alpha',0)\over (w-ct) -\check{z}(\alpha',0)} d \alpha'  \\
&  = u(w,t) - u(w-ct,0) = 0.  
\end{split}
\end{equation}
However, for a point $w_0=\check{z}(\alpha,0) + ct \in S (t)$, 
the Plemelj formulas state that 
$$
\lim_{w \to w_0^{\pm}} m(w) = \textrm{PV} {1 \over 2 \pi i} \int_\R {\gamma(\alpha',t) -\gamma(\alpha',0)\over\check{z}(\alpha,0) -\check{z}(\alpha',0)} d \alpha \pm {1 \over 2} {
\gamma(\alpha,t)-\gamma(\alpha,0) \over \check{z}_\alpha(\alpha,0)}
$$
where the ``$+$" and ``$-$" signs modifying $w_0$ in the limit indicate that the limit is taken from ``above" or ``below" $S (t)$, respectively. But, of course, $m$ is identically zero so that 
$$
{1 \over 2} \left(
\gamma(\alpha,t)-\gamma(\alpha,0)\right)= \pm \check{z}_\alpha(\alpha,0) 
\textrm{PV} {1 \over 2 \pi i} 
\int_\R {\gamma(\alpha',t) -\gamma(\alpha',0)\over\check{z}(\alpha,0) -\check{z}(\alpha',0)} d \alpha, 
$$
which in turn implies $\gamma(\alpha,t) = \gamma(\alpha,0)$. Since this is true for any $t$ and any $\alpha$, we see that $\gamma_t = 0$, the second equation in \eqref{TW v1}.

\end{proof}

\begin{remark} We additionally assume that $S (t)$ is parameterized to be proportional to arclength, {\it i.e.}
\begin{equation}\label{arclength param}
|z_\alpha| = \sigma= \text{constant} > 0
\end{equation}
for all $(\alpha,t)$. One may worry that the enforcement of the parameterization such that $z_t = c$ in \eqref{TW v1} is at odds with this sort of 
arclength parameterization. However, notice that $z_t = c$ implies that $z_{\alpha t} = 0$ which in turn implies that $z_\alpha$ (and thus $|z_\alpha|$) does not depend on time.
Then the reparamaterization of $S (t)$ given by $\widetilde{z}(\beta(\alpha),t) = z(\alpha,t)$ where $d \beta/d \alpha = |z_\alpha|/\sigma$ has $|\widetilde{z}_\beta| = \sigma$.   
Thus it is merely a convenience to assume \eqref{arclength param}.  
We will select a convenient choice for $\sigma$ later.
Arguments parallel to the above show that $z_t = c$ implies that $\theta_t = 0$ and thus we will view $\theta$ as being a function of $\alpha$ only.
\end{remark}

Now we insert the ansatz \eqref{TW v1} and the arclength parameterization \eqref{arclength param} into the equations of motion \eqref{z eqn} and \eqref{gamma eqn v1}. First, as observed in \cite{AAW}, 
we see that elementary trigonometry shows that $z_t = c$ and \eqref{z eqn} are equivalent to
\begin{equation}\label{U eqn}
U =- c \sin\theta  
\end{equation}
and
\begin{equation}\label{V eqn}
V=c \cos\theta .
\end{equation}
Notice this last equation selects $V$ in terms of the tangent angle $\theta$. That is to say \eqref{V eqn} should be viewed as the definition of $V$. On the other hand \eqref{U eqn} should be viewed as one of the equations we wish to solve. Using \eqref{this is U}, we rewrite it as
\begin{equation}\label{wn eqn}\Re(W^*N) + c \sin\theta  = 0.
\end{equation}

The above considerations transform \eqref{gamma eqn v1} to:
\begin{multline}\label{gamma eqn v2}
0=\tau\frac{\theta_{\alpha\alpha}}{\sigma}+\frac{\{(c\cos\theta-
\Re( W^* T))\gamma\}_{\alpha}}{\sigma}       \\
-2A\left(\frac{1}{8}\frac{(\gamma^{2})_{\alpha}}{\sigma^{2}}
+g\sin\theta -(c\cos\theta -
\Re(W^* T)
)
\Re(W^*_\alpha T)
\right).
\end{multline}
The last part of this expression may be rewritten as follows.  
Observe that 
\begin{equation}    \label{this}
-{1 \over 2}\partial_\alpha \{  (c \cos\theta -\Re(W^* T))^2  \}  
= (c \cos\theta  - \Re(W^*T))\left(c \sin\theta  \theta_\alpha +  \Re(W^* T_\alpha)+\Re(W^*_\alpha T) \right).
\end{equation}
Using \eqref{nice T}, we see that $T_\alpha = N \theta_\alpha$.  
Thus since $\theta$ is real valued and by virtue of \eqref{wn eqn}, we have 
$$c \sin\theta  \theta_\alpha +  \Re(W^* T_\alpha) = (c\sin\theta  + \Re(W^*N))\theta_\alpha  =  0.$$  
So \eqref{this} simplifies to 
$$
-{1 \over 2}\partial_\alpha (c \cos\theta -\Re(W^* T))^2= (c \cos\theta  - \Re(W^*T))\Re(W^*_\alpha T) .
$$
Hence 
\begin{multline}  \label {gamma**} 
0=\tau\frac{\theta_{\alpha\alpha}}{\sigma}+\frac{\{(c\cos\theta -
\Re( W^* T))\gamma\}_{\alpha}}{\sigma}\\
-2A\left(\frac{1}{8}\frac{(\gamma^{2})_{\alpha}}{\sigma^{2}}+
g \sin\theta +{1 \over 2}\partial_\alpha (c \cos\theta -\Re(W^* T))^2%
\right).
\end{multline}
which we rewrite as
\begin{multline}
-\theta_{\alpha \alpha}  =    \Phi(\theta,\gamma;c, \sigma) 
:=  {1 \over \tau}{(\partial_\alpha\{(c\cos\theta - \Re( W^* T))\gamma\})}\\
-       {A\over \tau}\left(\frac{1}{4\sigma}{\partial_\alpha(\gamma^{2})}
+      {2 g\sigma }\sin\theta +{\sigma}\partial_\alpha \{c \cos\theta -\Re(W^* T))^2\}  \right).
\end{multline}

Note that we have not specified $z$ as one of the dependencies of $\Phi$.  
This may seem unusual, given the prominent role of $z$
in computing the Birkhoff-Rott integral $W^*$. However, 
given $\sigma$  in \eqref{arclength param} one can determine $z(\alpha,t)$ solely from 
the tangent angle $\theta(\alpha)$, at least up to a rigid translation. 
Specifically, and without loss of generality, we have
\begin{equation}\label{reconstruct}
z(\alpha,0) = z(\alpha,t)-ct     =   \sigma \int_0^{\alpha}  e^{i \theta(\alpha')} d\alpha'.
\end{equation}
In this way, we view $W^*$ as being a function of $\theta$, $\gamma$ and $\sigma$. 

In short, we have shown the following:
\begin{lemma}\label{TWE}{\bf (Traveling wave equations, general version)}
Given functions $\theta(\alpha,t)$ and $\gamma(\alpha,t)$ and constants $c \in\R$ and $\sigma > 0$, compute $z(\alpha,t)$
from \eqref{reconstruct}, $W^*$ from \eqref{B} and $N$ and $T$ from \eqref{nice T}. If 
\begin{equation} \label{TW v3}
\Re(W^* N) + c \sin\theta  = 0 \quad \text{and} \quad \theta_{\alpha \alpha} + \Phi(\theta,\gamma;c,\sigma) = 0
\end{equation} 
holds then $(z(\alpha,t),\gamma(\alpha,t))$ is a traveling wave solution with speed $c$ for $\eqref{z eqn}$ and $\eqref{gamma eqn v1}$.

\end{lemma}

It happens that under the assumption that the traveling waves are spatially periodic,  \eqref{TW v3}
can be reformulated as ``identity plus compact" which, in turn will allow us to employ powerful abstract global bifurcation results.
The next section deals with how to deal with spatial periodicity.

\section{Spatial periodicity}
To be precise, by spatial periodicity we mean the following:  
\begin{definition}
Suppose that $(z(\alpha,t),\gamma(\alpha,t))$  is a solution of \eqref{z eqn} and \eqref{gamma eqn v1} such that
$$
S (t) = S (t) + M
$$
and
$$
u(w + M,t) = u(w,t)
$$
for all $t$ and $w \notin S (t)$, then the solution is said to be (horizontally) spatially periodic with period~$M$.
\end{definition}

It is clear if one has a spatially periodic curve $S (t)$ then it can be parameterized in such a way that the 
parameterization is $2\pi$-periodic in its dependence on the parameter. That is to say, 
the curve can be parameterized such that
\begin{equation}
\label{periodicity}
z(\alpha+2 \pi,t) = z(\alpha,t) +M.
\end{equation}
It is here that  we encounter a sticky issue. As described in Lemma \ref{TWE}, our goal is to find $\theta$ and $\gamma$ such that \eqref{TW v3} holds and additionally \eqref{periodicity} holds. 
The issue is that, given a function $\theta(\alpha)$ which is $2\pi-$periodic with respect to $\alpha$, 
it may not be the case that the curve $z$ reconstructed from it via \eqref{reconstruct} satisfies \eqref{periodicity}.
In fact, due to  \eqref{arclength param},    
 the periodicity \eqref{periodicity} is valid if and only if
\begin{equation} 
\label{reconstruct condition}
2\pi \overline{\cos\theta} :=
\int_0^{2\pi} \cos(\theta(\alpha')) d\alpha' ={M \over \sigma} \quad \text{and} \quad 
2\pi \overline{\sin\theta}   :=   \int_0^{2\pi} \sin(\theta(\alpha')) d\alpha' = 0.
\end{equation}
We could impose \eqref{reconstruct condition} on $\theta$. 
However, we follow another strategy which leaves $\theta$ free 
by modifying  \eqref{TW v3} so that  \eqref{reconstruct condition} holds.  

Indeed, we first fix the spatial period $M>0$.  
Suppose we are given a real $2\pi$-periodic function $\theta(\alpha)$ for which 
\begin{equation}\label{not vertical}
\overline{\cos\theta} \ne 0 
\end{equation}
so that the period $M$ of the curve will not vanish.  
Then we  define the ``renormalized curve" as 
\begin{equation}\label{this is z}  
\widetilde{Z} [\theta](\alpha)  =  
\frac{M}{2\pi \overline{\cos\theta}}  \left\{  
\int_0^\alpha e^{i\theta(\beta)} d\beta  -  i\alpha\ \overline{\sin\theta}  \right\}
\end{equation}
Of course, this function is one derivative smoother than $\theta$.  
A direct calculation shows that
\begin{equation}\label{periodicity 2}
\widetilde{Z}[\theta](\alpha+2\pi) = \widetilde{Z}[\theta](\alpha) + M.
\end{equation}
Thus $w=\widetilde{Z}[\theta]$ is the parameterization of a curve which satisfies 
\begin{equation}\label{per 2}
w(\alpha+2\pi) = w(\alpha) + M \text{ for all $\alpha$ in $\R$}.
\end{equation}  
Now $\ds \partial_\alpha\widetilde Z[\theta]   =   
\frac{M}{2\pi \overline{\cos\theta}}   (\exp(i\theta [\alpha])  -i\overline{\sin\theta})$, and 
the tangent and normal vectors for $\widetilde{Z}[\theta]$ are given by:
\begin{equation}\label{tangent def 2}
\widetilde{T} [\theta]:= {\partial_\alpha \widetilde{Z}[\theta]   /  \vert \partial_\alpha \widetilde{Z}[\theta] \vert}
\quad \text{ and } \quad
\widetilde{N} [\theta]:= {i \partial_\alpha \widetilde{Z}[\theta]   /   \vert \partial_\alpha \widetilde{Z}[\theta] \vert}.
\end{equation}
These expressions are not equal to $e^{i\theta}$ and $i e^{i \theta}$, as was the case for $T$ and $N$ in \eqref{nice T}. 

For a given real function $\gamma(\alpha)$       and parametrized curve $w(\alpha)$, 
define the {\it Birkhoff-Rott integral}  
\begin{equation*}\label{B2}
B[w]\gamma(\alpha):=\frac{1}{2 \pi i }\mathrm{PV}\int_{\R} {\gamma(\alpha') \over w(\alpha) -w(\alpha')} d \alpha'  .
\end{equation*}
Thus $W^*(\cdot,t)  =  B[z(\cdot,t)]\gamma(\cdot,t)$.  
								If $w$ satisfies \eqref{per 2}, 
we can rewrite this integral as  
\begin{equation}\label{B periodic 1}
B[{w}]\gamma(\alpha)=\frac{1}{2  i M}\mathrm{PV}\int_{0}^{2 \pi} 
\gamma(\alpha') \cot\left({\pi \over M}  (w(\alpha) -w(\alpha')) \right) d \alpha' 
\end{equation}
by means of Mittag-Leffler's famous series expansion for the cotangent (see, e.g., Chapter 3 of 
\cite{ablowitzFokas}).  
						Finally, 
for any real $2 \pi$-periodic functions $\theta$ and $\gamma$ and any constant $c \in \R$, define 
	\begin{multline} 
\widetilde{\Phi}(\theta,\gamma;c) 
:={1 \over \tau} \partial_\alpha \{c\cos\theta -
\Re( B[ \widetilde{Z}[\theta] ]\gamma\     \widetilde{T}[\theta]))\gamma  \}                                                                             \\
-{A\over \tau}\left(\frac{\pi \overline{\cos\theta}}{2M}     \partial_\alpha  {(\gamma^{2})}
+   \frac   {gM}   {\pi \overline{\cos\theta}       } 
\left(  \sin\theta  - \overline{\sin\theta} \right) 
+  \frac   {M}   {2\pi \overline{\cos\theta}}        \partial_\alpha 
\{   (c \cos\theta 
-\Re(B[ \widetilde{Z}[\theta] ]\gamma\, \widetilde{T}[\theta]))    ^2\}  
\right).
\end{multline}
						In terms of  
these definitions the basic equations are rewritten as follows:  
\begin{proposition}\label{TWE2}  
{\bf (Traveling wave equations, spatially periodic version)}
If the 
 $2\pi$-periodic functions  $\theta(\alpha)$, $\gamma(\alpha)$ and the constant $c \ne 0$  
 satisfy \eqref{not vertical} and 
\begin{equation} \label{TW v4}
\Re(B[ \widetilde{Z}[\theta] ]\gamma\, \widetilde{N}[\theta]) + c \sin\theta  = 0 \quad \text{and} \quad \theta_{\alpha \alpha} + \widetilde{\Phi}(\theta,\gamma;c) = 0
\end{equation} 
then
 $(\widetilde{Z}[\theta](\alpha)+ct,\gamma(\alpha,t))$ is a spatially periodic traveling wave solution with speed $c$ and period $M$ for $\eqref{z eqn}$ and $\eqref{gamma eqn v1}$.
\end{proposition}
\begin{proof}
Putting $w = \widetilde{Z}[\theta]$,  from the definitions above we have 
$\widetilde{N}[{\theta}]=iw_\alpha/|w_\alpha |$. 
Thus by  Lemma \ref{incompressibility} below,  
$$\displaystyle
\int_0^{2 \pi} \Re(B[ \widetilde{Z}[\theta] ]\gamma  \ \widetilde{N}[\theta])\ d\alpha = 0.$$
Together with the first equation in \eqref{TW v4} and the fact $c \ne 0$, this gives 
 \begin{equation}\label{win}
\overline{\sin \theta} = 0.\end{equation} 
Now we let $\sigma = M/(2\pi\overline{\cos\theta})$      and compute $z(\alpha,t)$
from \eqref{reconstruct}, $W^*$ from \eqref{B}, and $N$ and $T$ from \eqref{tangent def}.
By \eqref{reconstruct} and \eqref{this is z}  
 we see that
$z(\alpha,t) = \widetilde{Z}[\theta](\alpha)+ct$. This in turn gives 
$W^*=B[ \widetilde{Z}[\theta] ]\gamma$, $N=\widetilde{N}[\theta]$ and $T=\widetilde{T}[\theta]$. 
Together with the fact that $\overline{\sin\theta} = 0$, this shows that 
$$
\widetilde{\Phi}(\theta,\gamma;c)=\Phi(\theta,\gamma;c,\sigma).
$$
Thus both equations in \eqref{TW v4} coincide exactly their counterparts in \eqref{TW v3}. Proposition \ref{TWE} then shows that $z(\alpha,t)$ is a traveling
wave with speed $c\in\R$. We know that $z(\alpha,t)$ is $M-$periodic since it was constructed from $\theta$ with $\widetilde{Z}$.  
\end{proof}

\begin{lemma} \label{incompressibility} 
If $w(\cdot)$ satisfies \eqref{per 2} 
and $\gamma(\cdot)$ is a $2\pi$-periodic function,  then 
$$
\int_0^{2 \pi} \Re\left(   B[w]\gamma(\alpha)  {iw_\alpha(\alpha) \over | w_\alpha (\alpha)| } \right) d\alpha = 0.
$$
\end{lemma}
\begin{proof}  This lemma says that the mean value of the normal component of $B[w](\gamma)$ is equal to zero.  This follows from the fact that
$B[w](\gamma)$ extends to a divergence-free field in the interior of the fluid region, and from the Divergence Theorem.   
\end{proof} 

\section{Reformulation  as ``identity plus compact" }

\subsection{Mapping properties}
Let
$\H^s_\text{per}:=\H^s_{\text{per}}[0,2\pi]$ be the usual Sobolev space of $2\pi-$periodic functions from $\R$ to $\C$ whose first $s \in {\bf N}$ weak derivatives  are  square integrable.  
Likewise for intervals $I \subset \R$, let
 $\H^s(I)$ be the usual Sobolev space of functions from $I$ to $\C$  whose first $s \in {\bf N}$ weak derivatives  are  square integrable.  Finally,
 $\H^s_\text{loc}$ is the set of all functions from $\R$ to $\C$ which are in $\H^s(I)$ for all bounded intervals $I \subset \R$.
 
By \eqref{periodicity 2},  $\widetilde{Z}[\theta](\alpha)-M \alpha/2\pi$ is periodic. 
Let
 $$
 \H^s_M:=\left\{ w \in \H^s_\text{loc}: w(\alpha) - M \alpha/2 \pi  \in \H^s_\text{per} \right\}.
 $$
Clearly $\H^s_M$ is a complete metric space with the metric of $\H^s_\text{per}$
We have the following lemma concerning the renormalized curve $\widetilde{Z}[\theta]$:
\begin{lemma}\label{Z props}
For $s \ge 1$ and $h\ge0$ let 
$$
\U^s_h:=\left\{ \theta \in \H^s_\text{per} : \int_0^{2\pi} \cos(\theta(\alpha)) d\alpha > h \right\}.
$$
Then the map $\widetilde{Z}[\theta]$ defined in \eqref{this is z} is smooth  from $\U^s_h$ into $\H^{s+1}_M$
and
the maps $\widetilde{N}[\theta]$ and $\widetilde{T}[\theta]$ given in \eqref{tangent def 2}  are 
smooth from $\U^s_h$ into $\H^s_\text{per}$. 
Moreover, for any $h > 0$, there exists $C>0$ such that
\begin{equation}\label{Z estimate}
\| \widetilde{T} [\theta] \|_{\H^s_\text{per}} + \| \widetilde{N} [\theta] \|_{\H^s_\text{per}} 
+ \| \widetilde{Z}[\theta] \|_{\H^{s+1}(0,2\pi)} +
\left \| {1 \over \partial_\alpha \widetilde{Z}[\theta]} \right \|_{\H^{s}_\per}
\le C(1+ \| \theta \|_{\H^s_\text{per}}).
\end{equation}
for all $\theta \in \U^s_h$.
\end{lemma}

\begin{proof} 
As already mentioned, $\widetilde{Z}[\theta]$ is one derivative smoother than $\theta$.  
 A series of naive estimates leads to the bound on $ \widetilde{Z}[\theta] $ in \eqref{Z estimate}.  Next, since $\theta$ belongs to $\U^s_h$, 
$$
\overline{\sin\theta} + {1 \over 100 \pi^2} h^2    
\le {1 \over 2\pi} \left[ \int_0^{2\pi}\sin(\theta(a))da + {1 \over 50\pi}\left(\int_0^{2 \pi} \cos(\theta(a))da  \right)^2 \right].
$$
The Cauchy-Schwarz inequality on the cosine term leads to  
\begin{equation*}
\overline{\sin\theta}+ {1 \over 100 \pi} h^2    
\le   {1 \over 2\pi} \int_0^{2\pi}\left( \sin(\theta(a)) + {1 \over 25} \cos^2(\theta(a)) \right) da    \le 1  
\end{equation*}
since  $\ds \sin(x) + (1 / 25) \cos^2(x) \le 1$
Thus 
$$
\overline{\sin\theta} \le 1 -  {1 \over 100 \pi^2} h^2.
$$
For  $h >0$, this implies that $ e^{i\theta} - i\, \overline{\sin\theta}$ cannot vanish. 
Hence  $1/\partial_\alpha \widetilde{Z}[\theta] \in \H^s_\per$ and the remaining bounds in \eqref{Z estimate}
follow by routine estimates. 
The smooth dependence of $\widetilde{Z}$, $\widetilde{N}$ and $\widetilde{T}$  on $\theta$ 
is a consequence of standard results on compositions. 
\end{proof}

The most singular part of the Birkhoff-Rott operator $B$ is essentially the periodic Hilbert transform $H$, 
which is defined as  
$$
H \gamma(\alpha):= {1 \over 2 \pi } \mathrm{PV} \int_0^{2 \pi} \gamma(\alpha') \cot\left( {1 \over 2} (\alpha -\alpha') \right) d \alpha'.   $$
It is well-known that 
for any  $s \ge 0$, $H$ is a bounded linear map from $\H^s_\text{per}$ to $\H^s_{\per,0}$ 
(the subscript $0$ here indicates that the average over a period vanishes).
Moreover, 
$H$ annihilates the constant functions and 
$H^2 \gamma  = -\gamma + \overline\gamma$, where $\ds \overline{\gamma}=:{1 \over 2\pi} \int_0^{2\pi} \gamma(a) da.  $
In order to conclude that the leading singularity of the function $B[w]\gamma$ is given in terms of $H\gamma$, 
we require a ``chord-arc" condition, as stated in the following lemma.  
\begin{lemma}   
\label{hilbert plus smooth}
For $b \geq 0$ and $s \ge 2$, let the ``chord-arc space" be 
$$
\mathcal{C}_{b}^s:=\left\{ w(\alpha) \in \H^s_M :
\inf_{\alpha,\alpha' \in [0,2\pi]} \left \vert {w(\alpha')-w(\alpha) \over \alpha' - \alpha} \right \vert > b   \right\}
$$
and the remainder operator $K$ be 
$$
 {K}[w] \gamma(\alpha):=
B[w]\gamma(\alpha) - {1 \over 2   i w_\alpha(\alpha)} H \gamma(\alpha).$$
 						Then 
 $(w,\gamma) \mapsto {K}[w]\gamma$ is a smooth  map from
$\mathcal{C}^s_b  \times \H^1_\per \to \H^{s-1}_\per.$ If $b>0,$ 
then there exists a constant $C>0$ such that for all $w\in\mathcal{C}^{s}_{b}$ and for all
$\gamma\in\mathcal{H}^{1}_{per},$
$$
\| {K}[w] \gamma\|_{\H^{s-1}_\text{per}} \le C \| \gamma \|_{\H^1_\text{per}} \exp \left\{ C \| w \|_{\H^s(0,2\pi)}\right\}.
$$
\end{lemma}
\begin{proof}  See Lemma 3.5 of \cite{A}. We mention that related lemmas can be found elsewhere in the 
literature, such as in \cite{BHL}.
 \end{proof}

The set $\mathcal{C}_b^s$ is the open subset of $\H^s_M$ of functions whose graphs satisfy the ``chord-arc" condition. 
{\it This condition precludes self-intersection of the graph.} 
Note that this is true
even in the case where $b = 0$ since we have selected the strict inequality in the definition. 
Of course if $b>0$, membership of $w$ in this set $\mathcal{C}_b^s$ 
implies that $|w_\alpha(\alpha)| \ge b$ \ for all $\alpha$. 

Note that $H\gamma$ is real because $\gamma$ is real-valued.   
Also note that the definition of $\widetilde{T}[\theta]$ implies that 
$\widetilde{T}[\theta]/\partial_\alpha \widetilde{Z}[\theta] = 1/|\partial_\alpha \widetilde{Z}[\theta] |$ is also real. 
Thus 
\begin{equation}\label{WT is smooth}\begin{split}
\Re \left( (B[ \widetilde{Z}[\theta] ]\gamma\ \widetilde{T}[\theta]\right) 
&=\Re\left(  \left( K[\widetilde{Z}[\theta]]\gamma\right) \widetilde{T}[\theta] \right) + \Re\left(  \left( {1 \over 2  i \partial_\alpha \widetilde{Z}[\theta]} H \gamma  \right)
 \widetilde{T}[\theta] \right)\\
 &= \Re\left(  \left( K[\widetilde{Z}[\theta] ]\gamma\right) \widetilde{T}[\theta] \right) 
\end{split}
\end{equation}
and similarly
\begin{equation}\label{WN is not smooth}\begin{split}
\Re ( (B[ \widetilde{Z}[\theta] ]\gamma\ \widetilde{N}[\theta]) 
={1 \over 2   |\partial_\alpha \widetilde{Z}[\theta]|} H \gamma  
+ \Re(   K[\widetilde{Z}[\theta] ]\gamma\ \widetilde{N}[\theta] ) .
\end{split}
\end{equation}
Therefore, counting derivatives and applying Lemmas \ref{Z props} and  \ref{hilbert plus smooth}, 
we directly obtain the following regularity.  
\begin{cor} 
Let $s,s_1 \ge1$, $b > 0, h>0$ and 
$$
\U^s_{b,h}:= \left\{ \theta \in \U^s_h : \widetilde{Z}[\theta] \in \mathcal{C}^{s+1}_b \right \}.
$$
Then the mappings 
$(\theta,\gamma) \to B[ \widetilde{Z}[\theta] ]\gamma$ 
and $(\theta,\gamma) \to \Re( B[ \widetilde{Z}[\theta] ]\gamma\ \widetilde{N}[\theta] )$
are  smooth  from $\U^s_{b,h} \times \H^{s_1}_\text{per}$ into $\H^{\min\left\{s,s_1\right\}}_\text{per}$.  
Furthermore, 
$\Re( B[ \widetilde{Z}[\theta] ]\gamma\ \widetilde{T}[\theta] )$ 
is a smooth map from  $\U^s_{b,h} \times \H^1_\text{per}$
into $\H^{s}_\text{per}$.
\end{cor}

\begin{cor}    \label{zero mean}
$\widetilde{\Phi}(\theta,\gamma;c)$ is a smooth map from $\U^1_{b,h} \times \H^1_\text{per} \times \R$ into
$L^2_\text{per,0}:=\H^0_{\per,0} $.
\end{cor}
\begin{proof} The fact that $\widetilde{\Phi}(\theta,\gamma;c)$ is a smooth map from $\U^1_{b,h} \times \H^1_\text{per} \times \R$
into $L^2_\per$ follows from the previous corollary and the definition of $\widetilde{\Phi}$. Examination of the terms in $\widetilde{\Phi}$ shows that all but one is a perfect derivative, and
thus will have mean value zero on $[0,2\pi]$. The remaining term is a constant times 
$\sin\theta   - \overline{\sin\theta}$, which also has mean zero. 
Thus  $\widetilde{\Phi} \in L^2_{\per,0}.$
\end{proof}

We introduce the ``inverse" operator
$$
\paa f (\alpha):= \int_0^\alpha \int_0^a f(s) ds da - {\alpha  \over 2 \pi} \int_0^{2 \pi} \int_0^a f(s) ds da,      $$
which is bounded from $H^s_{per,0}$ to $H^{s+2}_{per}$.  
Indeed, it is obvious that $\partial_\alpha^2 (\paa f)= f$, 
so we only need to demonstrate the periodicity of $\paa f$ for any $f\in H^s_{per,0}$.  
To this end, we compute  
\begin{equation*}
\begin{split}
\paa f(\alpha+2\pi)&=  \int_0^{\alpha + 2\pi} \int_0^a f(s) ds da - {\alpha  + 2 \pi \over 2 \pi} \int_0^{2 \pi} \int_0^a f(s) ds da 
\\&=
 \int_{2 \pi}^{\alpha + 2\pi} \int_0^a f(s) ds da 
- {\alpha  \over 2 \pi} \int_0^{2 \pi} \int_0^a f(s) ds da. 
\end{split}
\end{equation*}
\begin{equation*}
\begin{split}
= \int_{0}^{ \alpha} \int_0^{b} f(s) ds db 
- {\alpha  \over 2 \pi} \int_0^{2 \pi} \int_0^a f(s) ds da = \paa f(\alpha).
\end{split}
\end{equation*}

\subsection{Final reformulation}

Using \eqref{WN is not smooth} in the first equation of $\eqref{TW v4}$  yields  the equation  
$$
H \gamma + 2 \left \vert \partial_\alpha \widetilde{Z}[\theta] \right \vert 
\Re \left( ( K[ \widetilde{Z}[\theta]  ] \gamma)\widetilde{N}[\theta]  \right) +  2 c \left \vert \partial_\alpha \widetilde{Z}[\theta] \right \vert \sin\theta  = 0.   $$
It will be helpful to break $\gamma$ up into the sum of its average value and a mean zero piece, so we let 
$$
  \gamma_1:=\gamma - \overline\gamma .   $$
Applying $H$ to both sides and using $H^2 \gamma  = -\gamma + \overline\gamma = -\gamma_1$, we obtain 
\begin{equation}\label{gamma 1 eqn}
 \gamma_1 - H \left \{      2  \left \vert \partial_\alpha \widetilde{Z}[\theta] \right \vert 
\Re \left( ( K[ \widetilde{Z}[\theta]  ]  (\overline\gamma  + \gamma_1) ) \widetilde{N}[\theta] \right) +  2 c \left \vert \partial_\alpha \widetilde{Z}[\theta] \right \vert \sin\theta    \right\}  = 0.
\end{equation}
 It will turn out that we are free to specify $\overline\gamma $ in advance, and so henceforth we will view $\overline\gamma $ as a constant in the equations, akin to $g$, $M$, $A$ or $\tau$. 

Now one of the equations we wish to solve is $\theta_{\alpha \alpha} + \widetilde{\Phi}(\theta,\gamma;c)=0$. 
We use $\paa$  to ``solve" this equation for $\theta$.
Keeping in mind that $\gamma = \overline\gamma  + \gamma_1$, we define   
\begin{equation}  \label{captheta}
\Theta(\theta,\gamma_1;c): = - \paa \widetilde{\Phi}(\theta,\overline\gamma +\gamma_1;c) .  
\end{equation}
Then 
the second equation in \eqref{TW v4}  is equivalent to  
$\ds  \theta - \Theta(\theta,\gamma_1;c) = 0  $.   
					Knowing 
that $\theta = \Theta$, we are also free to rewrite \eqref{gamma 1 eqn} as 
$\gamma_1- \Gamma (\theta,\gamma_1;c) = 0$, 
where
\begin{multline}       \label{Gamma}
\Gamma (\theta,\gamma_1;c) \\:=  2  H \left \{ 
\left\vert \partial_\alpha \widetilde{Z}[\Theta(\theta,\gamma_1;c)] \right \vert 
\Re \left( \left(K\left[ \widetilde{Z}[\Theta(\theta,\gamma_1;c)] \right](\overline\gamma +\gamma_1) \right) 
\widetilde{N}[\Theta(\theta,\gamma_1;c)]  \right)
\right. \\ \left. + c\left \vert \partial_\alpha \widetilde{Z}[\Theta(\theta,\gamma_1;c)] \right \vert \sin(\Theta(\theta,\gamma_1;c))           \right\}.
\end{multline}

Summarizing, our equations now have the form 
\begin{equation}   \label{TW v5}
\ds  \theta - \Theta(\theta,\gamma_1;c) = 0, \quad   \gamma_1- \Gamma (\theta,\gamma_1;c) = 0
\end{equation}
					The set where the solutions will be situated is  $\U=\U_{0,0}$ where 
\begin{multline} \label{solution set}
\U_{b,h}:=
\left\{ (\theta,\gamma_1;c) \in \H^1_\per \times \H^1_{\per,0} \times\R: \right. \\     \left.
\theta \text{ is odd, } \gamma_1 \text { is even, } 
 \overline{\cos\theta} > h, 
\widetilde{Z}[\theta] \in \CA^2_b \text{ and }  \widetilde{Z}[\Theta(\theta,\gamma_1;c)] \in \CA^{3}_b
\right\}.  
\end{multline}
These sets are given the topology of $\H^1_\per \times \H^1_{\per,0} \times\R$.  
Note that they are defined so that the functions have one derivative.  The following theorem states 
that $\Theta$ and $\Gamma$ gain an extra derivative.  
\begin{theorem}\label{TWE5}{\bf (``Identity plus compact" formulation)}
For all $b, h>0$, the pair $(\Theta,\Gamma )$ is a compact map from $\U_{b,h}$ into $\H^2_{\per,\text{odd}} \times \H^2_{\per,0,\text{even}}$
and is smooth from $\U$ into $\H^2_{\per,\text{odd}} \times \H^2_{\per,0,\text{even}}$.
If $(\theta,\gamma_1;c) \in \U$ solves  \eqref{TW v5}, 
then the pair 
 $(\widetilde{Z}[\theta](\alpha)+ct,\overline\gamma +\gamma_1(\alpha,t))$ is a spatially periodic, 
 symmetric traveling wave solution of $\eqref{z eqn}$ and $\eqref{gamma eqn v1}$ with speed $c$ and period $M$.
 \end{theorem}
\begin{proof}  
Observe that from the results of the previous section, 
$\Theta(\theta,\gamma_1;c)$ is a smooth map from $\U^1_{b,h} \times \H^1_{\per,0} \times \R$ into $\H^2_\per$ 
for any  $b,h > 0$.  
Careful unraveling of the definitions shows that  
$\Gamma (\theta,\gamma_1;c)$ is a smooth map from the set 
$\{ (\theta,\gamma_1;c) \in \U^1_{b,h} \times \H^1_{\per,0}  \times \R:
\widetilde{Z}[\Theta(\theta,\gamma_1;c)] \in \CA^{3}_b   \}$ 
into $\H^2_{\per,0}$. These facts, together with the uniform bound for fixed $b>0$ on the remainder operator $K$ in Lemma \ref{hilbert plus smooth}, imply that 
the mapping $(\Theta,\Gamma)$ is compact from $\U_{b,h}$ into  $\H^2_{\per,\text{odd}} \times \H^2_{\per,0,\text{even}}$, since $\H^2_{\per}$ is compactly embedded in $\H^1_{\per}$.
We conclude that $\Theta$ and $\Gamma $ are also smooth maps, but not necessarily compact, on the 
union of the previous sets over all $b,h>0$, which is to say 
that $(\Theta,\Gamma)$ is smooth on $\U$.  
The second statement in the theorem is obvious from the previous discussion.  
Finally, the subscripts ``odd" and ``even" in the target space for $(\Theta,\Gamma)$ above simply denote the subspaces which consist of odd and even functions.
That $(\Theta,\Gamma)$ preserves this symmetry can be directly checked from its definition; the computation is not short, but neither is it difficult. So we omit it.
\end{proof}

 \section{Global Bifurcation}
 \subsection {General considerations}\label{gen con}
Our basic tool is the following global bifurcation theorem, which is based on the use of Leray-Schauder degree.  
It is fundamentally due to Rabinowitz \cite{R} and was later generalized by Kielh\"ofer \cite{K}.  
\begin{theorem} \label{glbif} {\bf (general bifurcation theorem)}  
Let $X$ be a Banach space and  $U$ be an open subset of $X\times\mathbf{R}$.
Let $F$ map $U$ continuously into $X$.  
Assume that 
\begin{enumerate} [(a)] 
\item the Frechet derivative  $D_{\xi}F(0,\cdot)$ belongs to $C(\mathbf{R},L(X,X))$,   
\item  the mapping $(\xi,c)\to F(\xi,c)-\xi$ is compact from $X\times\R$ into $X$,   
and  \item  $F(0,c_0)=0$ and $D_{x}F(0,c)$ has an odd crossing number   at $c=c_{0}.$   
\end{enumerate} 
Let 
$\mathcal{S}$ denote the closure of the set of nontrivial solutions of $F(\xi,c)=0$ in 
$X\times\mathbf{R}.$  
Let $\mathcal{C}$ denotee the
connected component of $\mathcal{S}$ to which $(0,c_{0})$ belongs.  
Then one of the following alternatives is valid: 
\begin{enumerate}[(i)]
\item $\mathcal{C}$ is unbounded; or 
\item $\mathcal{C}$ contains a point $(0,c_{1})$ where $c_{0}\neq c_{1}$; or 
\item $\mathcal{C}$ contains a point on the boundary of $U.$
\end{enumerate}
\end{theorem}

{The {\it crossing number} is the number of eigenvalues of $D_xF(0,c)$ that pass through 
$0$ as $c$ passes through~$c_0$.}  
In his original paper  \cite{R} Rabinowitz assumed that $F$ has the form 
$F(\xi,c) = \xi - cG(\xi)$.  Kielh\"offer's  book \cite{K} permits the general form as above.   
Theorem II.3.3 of \cite{K} states Theorem \ref{glbif} in the case that $U=X\times\R$.  
The proof of Theorem \ref{glbif} with an open set $U$ is practically identical to that in \cite{K}.  

We apply this theorem to our problem by fixing $b,h>0$ and setting   $U = \U_{b,h}$, $X=\H^1_{\per,\text{odd}} \times \H^1_{\per,0,\text{even}}$, $\xi = (\theta,\gamma_1)$, and 
$F(\xi, c)  = \xi - (\Theta(\xi,c),\Gamma(\xi,c))$.  Then
the problem laid out in Theorem \ref{TWE5} fits into the framework of this theorem. 
All we need to do is to choose $c_0$ so that the linearization has an odd crossing number when $c = c_0$.  
In fact, the simplest case with crossing number one will suffice.  

\subsection{Computation of the crossing number}
This calculation is difficult primarily due to the large number of terms we must differentiate. 
Thus we introduce some notation which will help to compress the calculations.
For any map $\mu(\theta,\gamma_1;c)$, 
we use $(\thetab, \gammab)$ to denote the direction of differentiation. To wit, we define:
\begin{multline}
\mu_0:=\mu(0,0;c) \quad \text{and}\\   
D\mu:= D_{\theta,\gamma_1} \mu(\theta,\gamma_1;c) \big\vert_{(0,0;c)} (\thetab,\gammab)   
:=\lim_{\ep \to 0} {1 \over \ep}\left( \mu(\ep \thetab,\ep \gammab;c)-\mu(0,0;c)\right).
\end{multline} 
We let $Q(\theta,\gamma_1) := (\overline\gamma  + \gamma_1)^2$, 
$Y(\theta) = \overline{\sin\theta}$, $\Sigma(\theta) = M / (2\pi \overline{\cos\theta})$, 
and 
$\widetilde{\mathsf{W}}^*[\theta,\gamma_1]:=  
B[\widetilde{Z}[\theta]](\overline\gamma  + \gamma_1)$.
It is to be understood that by $\sin$ and $\cos$ we mean the maps $\theta \to \sin\theta $ and $\theta \to \cos\theta $, respectively.
We will first compute the linearizations of $\Theta$ and $\Gamma $ 
while ignoring the restrictions to symmetric (even/odd) functions;  
of course, computation of the full linearization will restrict in a natural way to the 
linearization of the symmetric problem.

The following quantities and derivatives thereof are elementary.  
$$
\sin_0 = 0,  \quad \cos_0 = 1,\quad Q_0= \overline\gamma^2, \quad 
Y_0 = 0,  \quad \Sigma_0 =1,
\quad \widetilde{Z}_0= {M\over 2\pi} \alpha,
$$
$$
\quad (\partial_\alpha \widetilde{Z})_0= \left \vert \partial_\alpha \widetilde{Z} \right \vert_0 = {M \over 2\pi},
\quad
\widetilde{T}_0= 1, 
\quad
\widetilde{N}_0 = i
\quad
 \text{and} \quad
 \widetilde{\mathsf{W}}^*_0 = 0.$$
$$
D\sin  = \thetab, \quad D \cos = 0, \quad DQ = 2 \overline\gamma  \gammab,\quad
D Y  = {1 \over 2\pi} \int_0^{2 \pi} \thetab(a) da, \quad D \Sigma = 0,
$$
$$
D\widetilde{Z}  ={i M \over 2 \pi} \left(
\int_0^\alpha \thetab(a) da - {\alpha \over 2\pi} \int_0^{2\pi} \thetab(a) da 
\right),
\quad
D\partial_\alpha \widetilde{Z}={i M \over 2 \pi} \left(
\thetab  - {1 \over 2\pi} \int_0^{2\pi} \thetab(a) da 
\right),
$$
$$
D\left \vert \partial_\alpha \widetilde{Z}\right \vert = 0,\quad
D\widetilde{T}= i \left( \thetab - {1 \over 2\pi} \int_0^{2 \pi} \thetab(a) da \right)
\quad \text{and} \quad
D\widetilde{N} = -\thetab + {1 \over 2\pi} \int_0^{2 \pi} \thetab(a) da. 
$$

The computation of
 $D\widetilde{\mathsf{W}}^*$ is somewhat more complicated.
By the product and chain rules, 
\begin{equation}
\begin{split}
D\widetilde{\mathsf{W}}^*&= D\left[ B[\widetilde{Z}[\theta]] (\overline\gamma  + \gamma_1) \right] 
(\thetab, \gammab) \\
& ={1 \over 2  i M} \mathrm{PV} \int_0^{2 \pi}
D \left[( \overline\gamma  + \gamma_1(\alpha')) \cot \left( {\pi \over M} \left( \widetilde{Z}[\theta](\alpha) - \widetilde{Z}[\theta](\alpha')\right)\right) \right]d\alpha' 
\\
&={1 \over 2  i M} \mathrm{PV} \int_0^{2 \pi}
\gammab(\alpha') \cot \left( {\pi \over M} \left( \widetilde{Z}_0(\alpha) - \widetilde{Z}_0(\alpha')\right)\right)d\alpha' 
\\
&-{\pi \over 2  i M^2} \mathrm{PV} \int_0^{2 \pi}
 \overline\gamma  \csc^2 \left( {\pi \over M} \left( \widetilde{Z}_0(\alpha) - \widetilde{Z}_0(\alpha')\right)\right) 
 \left(D \widetilde{Z}(\alpha) - D \widetilde{Z}(\alpha') \right) (\thetab) \ d\alpha' .
\end{split}
\end{equation}
Now we use the fact that $\ds \widetilde{Z}_0(\alpha) = {M / 2 \pi}\alpha$ and the definition of $H$ to 
see that the first of the two terms above is exactly 
$\ds {(\pi / i M)} H \gammab.$ 
The second term $T_2$ is 
\begin{equation}
\begin{split}
&-{ \pi \overline\gamma  \over 2  i M^2} \mathrm{PV} \int_0^{2 \pi}
  \csc^2 \left( {1 \over 2} \left( \alpha - \alpha'\right)\right) \left(D\widetilde{Z}(\alpha) - D\widetilde{Z}(\alpha')\right)d\alpha' 
\\
 =&-{ \pi\overline\gamma    \over   i M^2} \mathrm{PV} \int_0^{2 \pi}
 \cot \left( {1 \over 2} \left( \alpha - \alpha'\right)\right)  {\partial \over \partial \alpha'} D
\widetilde{Z}(\alpha')d\alpha'    
 =-{2 \pi^2 \overline\gamma  \over i M^2} H  \left( \partial_\alpha D\widetilde{Z} \right)
 \end{split}
\end{equation}
But 
$$\ds
 { \partial_\alpha} D\tilde{Z}
 =
 {i M \over 2 \pi} \left( \thetab(\alpha') -{1 \over 2\pi} \int_0^{2\pi} \thetab(a) da \right).
$$
Since $H$ annihilates constants,  the second term $T_2$ is equal to
$-\ds 
({\pi \overline\gamma  / M}) H \thetab.
$
Thus  
\begin{equation}  \label{DW}
D\widetilde{\mathsf W}^*= {\pi \over i M} H \gammab - { \pi \overline\gamma  \over  M} H \thetab.
\end{equation} 

In order to evaluate $D\Theta$, we have $\Theta_0 = 0$ and 
\begin{multline}
D \Theta = -{1 \over \tau} \paa \partial_\alpha
 D\left[
  (c\cos - \Re( \widetilde{\mathsf W}^* \widetilde{T}))(\overline\gamma  + \gamma_1)\right]\\+ {A\over \tau}\paa D \left[\frac{\partial_\alpha Q}{4\Sigma}
+2 g \Sigma \left( \sin - Y\right)
+{\Sigma}\partial_\alpha (c \cos-\Re(\widetilde{\mathsf W}^* \widetilde{T}))^2
\right]
\end{multline}
Carrying out $D$, we have  
\begin{equation*}
\begin{split}
D \Theta =& -{1 \over \tau}\paa\partial_\alpha
 \left[
  (cD \cos - \Re( (D \widetilde{\mathsf W}^*) \widetilde{T}_0)  -  \Re(  \widetilde{\mathsf W}_0^* (D\widetilde{T}))  
  )\overline\gamma 
  +  (c\cos_0 - \Re( \widetilde{\mathsf W}_0^* \widetilde{T}_0))\gammab\right]
  \\&+ {A\over \tau}\paa  \left[\frac{ \Sigma_0 \partial_\alpha DQ - \partial_\alpha (Q_0) D\Sigma }{4\Sigma^2_0} 
+2 g \Sigma_0 \left( D \sin - D Y\right)+ 2 g D \Sigma \left( \sin_0 - Y_0 \right) \right]
\\&
+ {A\over \tau}\paa \left[
(D\Sigma)\partial_\alpha (c \cos_0-\Re(\widetilde{\mathsf W}_0^* \widetilde{T}_0))^2  \right] \\&+
{A\over \tau}\paa  \left[2 \Sigma_0\partial_\alpha[ (c \cos_0-\Re(\widetilde{\mathsf W}_0^* \widetilde{T}_0))
(c D\cos-\Re((D\widetilde{\mathsf W}^*) \widetilde{T}_0)-\Re(\widetilde{\mathsf W}_0^* (D\widetilde{T}))) 
]
\right]
\end{split}
\end{equation*}
				When we use the expressions at the start of this section, this quantity reduces to 
\begin{multline}
D \Theta = -{1 \over \tau}\paa\partial_\alpha
 \left[
  - \overline\gamma  \Re( D \widetilde{\mathsf W}^* ) 
  +  c\gammab\right]
  \\+ {A\over \tau}\paa \left[\frac{ \overline\gamma  \partial_\alpha\gammab }{2(M/2\pi)} 
+
{g M \over \pi} P\thetab
 \right]
+ {A\over \tau}\paa \left[
 {M \over \pi} \partial_\alpha[ -c 
\Re(D\widetilde{\mathsf W}^* )
]
\right], 
\end{multline}
where 
$$
P \thetab := \thetab-{1 \over 2 \pi}\int_0^{2\pi} \thetab (a) da.
$$
				Using \eqref{DW} in this expression, we get 
\begin{equation}   \label{DTheta} 
D \Theta =-{\pi \overline\gamma  \over M} \left( {\overline\gamma  \over \tau} 
-{ c A M \over \pi \tau}
\right)
\paa \partial_\alpha
H \thetab
+{A g M \over \pi \tau} \paa P \thetab
  + \left( {A \overline\gamma  \pi \over \tau M} -{ c \over \tau}\right)\paa \partial_\alpha \gammab.
\end{equation}  

Lastly, to compute compute $D\Gamma $, we have $\Gamma_{0} = 0$ and 
by \eqref{WN is not smooth} and \eqref{Gamma}, 
$$
\Gamma =
2 H \left( \left \vert \partial_\alpha \widetilde{Z}[\Theta] \right \vert \Re\left(
\widetilde{W}^*\left[\Theta,\overline\gamma  + \gamma_1\right] \widetilde{N}\left[ \Theta \right] \right)
-{1 \over 2 } H \gamma_1\right)  +2  c H\left( \left \vert \partial_\alpha \widetilde{Z}[\Theta] \right \vert \sin(\Theta)
\right)
$$
Differentiating, we get 
\begin{equation}
\begin{split}
D \Gamma  = &
2 H   
\left\{  D \left \vert \partial_\alpha \widetilde{Z} \right \vert \circ D\Theta\  \Re\left(
\widetilde{\mathsf W}^*_0 \widetilde{N}_0 \right) \right. 
 +\left \vert \partial_\alpha \widetilde{Z}_0 \right \vert \Re\left(
D \left(\widetilde{W}[\Theta,\overline\gamma +\gamma_1] \right)\widetilde{N}_0\right)   \\ &\left. 
 +\left \vert \partial_\alpha \widetilde{Z}_0\right \vert \Re\left(
\widetilde{\mathsf W}^*_0\  D \widetilde{N} \circ D \Theta \right) 
-{1 \over 2 } H \gammab\right\}    
+2  c H\left\{
 D \left \vert \partial_\alpha \widetilde{Z} \right \vert \circ D\Theta\  \sin_0 +
 \left \vert \partial_\alpha \widetilde{Z}_0 \right \vert D\sin \circ D \Theta   \right \}    
\\
= & \gammab+
{M \over \pi}H  \Re \left( i D \left(\widetilde{\mathsf W}[\Theta,\overline\gamma +\gamma_1] \right)
\right)  +{c M \over \pi} H D \Theta
\end{split}
\end{equation}
because $D|\partial_\alpha\widetilde Z| = 0$   and     $\widetilde W_0^*=0$.  
By \eqref{DW},  
we have 
$\ds
D \left(\widetilde{ W}[\Theta,\overline\gamma +\gamma_1] \right)
= {\pi \over iM} H \gammab - {\pi \overline\gamma  \over M} H D\Theta $, 
so that 
\begin{equation}  \label{DGamma}  
 D \Gamma  = \gammab - \Re (\gammab  -  i\overline{\gamma} D\Theta)  +  \frac{cM}{\pi} HD\Theta  =  
 {cM \over \pi} H D\Theta. 
 \end{equation}
 
Combining 
\eqref{DW}, \eqref{DTheta} and \eqref{DGamma}, we see that the linearization of the mapping 
$(\theta,\gamma_1) \to (\theta-\Theta,\gamma_1 -\Gamma )$ at $(0,0;c)$ is  
\begin{equation}\label{Lc}\begin{split}
L _c
\left[
\begin{array}{c} \thetab \\ \gammab 
\end{array}
\right]& :=
\left[
\begin{array}{c}
\thetab - D\Theta\\
\gammab-D\Gamma 
\end{array}
\right]\\&
= \left[
\begin{array}{cc}
1 + {\pi \overline\gamma  \over M} \left( {\overline\gamma  \over \tau} 
-{ c A M \over \pi \tau}
\right)
\paa\partial_\alpha
H 
- {AgM \over \pi \tau} \paa P
& 
  - \left( {A \overline\gamma  \pi \over \tau M} -{ c \over \tau}\right)\paa \partial_\alpha \\
{\overline\gamma  c } \left( {\overline\gamma  \over \tau} 
-{ c A M \over \pi \tau}
\right)
H \paa \partial_\alpha
H - {c M^2Ag \over  \pi^2  \tau} H\paa P
 &   1 
  -c \left( {A \overline\gamma   \over \tau } -{ cM \over \pi \tau}\right)H\paa \partial_\alpha 
\end{array}
\right]\left[
\begin{array}{c} \thetab \\ \gammab 
\end{array}
\right]
\end{split}
\end{equation}
Our goal is to find those values of $c$ such that $L_c$ has a one-dimensional nullspace.  
Because  we are working with $2\pi$-periodic functions, we may expand them as 
$$
\thetab(\alpha) = \sum_{k  \in \Z} \widehat{\thetab}(k) e^{ik \alpha}\quad \text{and} \quad \gammab(\alpha) = \sum_{k  \in \Z'} \widehat{\gammab}(k) e^{ik \alpha}.
$$
We have denoted $\Z' : = \Z / \left\{ 0 \right\}$.  
We can eliminate the $k = 0$ coefficient for $\gammab$ since it has zero mean.  
The operators $\partial_\alpha$, $H$, $P$ and $\paa$ can be represented on the Fourier side in the usual way:
$$
\widehat{\partial_\alpha \mu}(k) = ik \widehat{\mu}(k),\quad \widehat{H \mu}(k) = -i\sgn(k) \widehat{\mu}(k)
$$
$$
\quad \widehat{P \mu}(k) = (1 - \delta_0(k))\widehat{\mu}(k)  
\quad \text{and} \quad \widehat{\paa \mu} (k) = -{1 \over k^2} \widehat{\mu}(k), 
$$
where $\delta_0(k) = 1$ for $k = 0$ and is otherwise zero.  
Thus  $L_c$ is represented on the frequency side as the Fourier multiplier 
\begin{equation}\begin{split}
\widehat{L _c
\left[
\begin{array}{c} \thetab \\ \gammab 
\end{array}
\right]} (k) & =
\widehat{L _c}(k)
\left[
\begin{array}{c}\widehat{ \thetab }(k)\\ \widehat{\gammab}(k)
\end{array}
\right] \end{split}
\end{equation}
where, for $k \ne 0$
\begin{equation}
\begin{split}
\widehat{L _c}(k)
 &
= \left[
\begin{array}{cc}
1 - {\pi \overline\gamma  \over M} \left( {\overline\gamma  \over \tau} 
-{ c A M \over \pi \tau}
\right)
|k|^{-1} + {A gM \over \pi \tau} k^{-2}
& 
 i\left( { A \overline\gamma  \pi \over \tau M} -{ c \over \tau}\right)k^{-1} \\
i{\overline\gamma  c } \left( {\overline\gamma  \over \tau} 
-{ c A M \over \pi \tau}
\right)
 k^{-1} -i {c M^2 A g  \over \pi^2 \tau}  \sgn(k) k^{-2}
 &   1 
  +c \left( {A \overline\gamma   \over \tau } -{ cM \over \pi \tau}\right)|k|^{-1}
\end{array}
\right]\end{split}
\end{equation}
and
$  \widehat{L_c}(0)  $   is the identity. 

Now we can easily compute the point spectrum of $L_c$.   
In particular, $\lambda \in \C$ is an eigenvalue of $L_c$ if and only if
$\lambda$ is an eigenvalue of $\widehat{L_c}(k)$ for some integer $k$. 
For any nonzero integer $k$, 
an elementary computation shows that the two eigenvalues 
of $\widehat{L_c}(k)$  are $1$ and
$$
\lambda_k(c) := 1+ \frac{2\overline\gamma  c A M \pi -M^2c^2 - \overline\gamma ^2 \pi^2}{M\pi \tau} \left \vert k \right \vert^{-1} +\frac{g A M}{\pi \tau} \left \vert k \right \vert^{-2}
$$
Since this expression is even in $k$, every eigenvalue of $L_c$ has even multiplicity. 
So any crossing number for $L_c$ will  necessarily be even. 
However, the eigenvector of $\widehat{L_c}(k)$ associated to this eigenvalue is
$\ds
\left[ \begin{array}{c}
\sgn(k) {i \pi /c M} \\
1
\end{array}
\right]
$
which in turn implies that 
$$
\left[ \begin{array}{c}
 {i \pi /c M} \\
1
\end{array}
\right] e^{i k \alpha}
\quad \text{and} \quad
\left[ \begin{array}{c}
-{i \pi /c M} \\
1
\end{array}
\right] e^{-i k \alpha}
$$
are the corresponding eigenfunctions for $L_c$ with eigenvalue $\lambda_k(c)$.   
Of course, we can break these up into real an imaginary parts to get an equivalent basis for the eigenspace, 
namely, 
$$
\left[ \begin{array}{c}
 -{(\pi /c M)} \sin(k \alpha) \\
\cos(k \alpha)
\end{array}
\right]
\quad \text{and} \quad
\left[ \begin{array}{c}
{ (\pi /c M)} \cos(k \alpha) \\
\sin(k \alpha)
\end{array}
\right]. 
$$
Only the first of these satisfies the symmetry properties ($\theta$ odd, $\gamma_1$ even) 
required by our function space \eqref{solution set}. 
Thus when we take account of the symmetry, the dimension of the eigenspace equals one.  
We summarize the spectral analysis as follows.  
\begin{proposition} {\bf (Spectrum of $L_c$)}
Let $L_{c}$ be the linearization of the mapping  
$(\theta,\gamma,c) \in \U \to (\theta-\Theta,\gamma-\Gamma) \in \H^2_\per \times \H^2_{\per,0}$ at $(0,0;c)$.
The spectrum of $L_{c}$ consists of $1$ and the point spectrum
$$
\sigma_{\text{pt}}:=\left\{ \lambda_k(c): = 1+ \frac{2\overline\gamma  c A M \pi -M^2c^2 - \overline\gamma ^2 \pi^2}{M\pi \tau} k^{-1} +\frac{g A M}{\pi \tau} k^{-2}
: k \in \N
\right\}. 
$$
Moreover, each eigenvalue $\lambda$ has geometric and algebraic multiplicity $N_\lambda(c)$ where $$
N_\lambda(c):= \left \vert 
\left\{ k \in \N \text{ such that } \lambda_k (c)= \lambda \right\}
\right \vert.
$$
The eigenspace for $\lambda$ is  
$$
E_\lambda(c):=\text{span} \left\{
\left[ \begin{array}{c}
 -{(\pi /c M)} \sin(k \alpha) \\
\cos(k \alpha)
\end{array}
\right]:  \text{$k \in \N$ such that }\lambda_k(c) = \lambda
\right\}.
$$
\end{proposition}
\begin{cor}\label{crossing}
Fix $A$, $g$, $\overline\gamma  \in \R$ and $\tau,M>0$.      Let
$$
\mathcal{K}:=\left\{ k \in \Z :
\pi^2 \overline\gamma ^2 A^2 k ^2 + \pi \tau k^3 M - \pi^2 k^2 \overline\gamma ^2 + k A g M^2 > 0 \text{ and }
A g M/\pi \tau k \notin \N/\left\{ k \right\}  \right\}.  $$
For $k \in \N$, let 
$$
c_\pm(k):={\pi \overline\gamma  A \over M} \pm {1 \over kM} \sqrt{\pi^2 \overline\gamma ^2 A^2 k ^2 + \pi \tau k^3 M - \pi^2 k^2 \overline\gamma ^2 + k A g M^2}.  $$
Then  $\left \vert \mathcal{K} \right \vert = \infty$ 
and  $L_{c}$ has  crossing number equal to one at a real value $c = c_\pm(k)$ 
if and only if $k \in \mathcal{K}$.
If $A=0$, then $\mathcal K=\Z$.  
\end{cor}
\begin{proof}
Fix $k \in \N$. 
We are looking for the real values of $c\in\R$ for which $\lambda_k=0$.  
There are at most two roots.   
A routine calculation shows that $\lambda_k(c) = 0$ if and only if $c = c_\pm(k)$.   
The first condition in the definition of $\mathcal{K}$ shows that $c_\pm(k)$ is a real number. 
Thus  $\lambda =0$ is an eigenvalue.
We must compute its crossing number and the first step is to calculate its multiplicity, denoted $N_0(c_\pm(k))$. 
Thus, given $c=c_\pm(k)$,  we must find all $l \in \N$
such that $\lambda_l(c_\pm(k)) = 0.$ Clearly $l = k$ works. 
A small amount of algebra shows that the only other possible solution of $\lambda_l(c_\pm(k)) = 0$ 
is $$l = l_k:= AgM/\pi\tau k.$$

Another calculation shows that
$$
\lambda_k(c_\pm(k) + \ep) = \mp {2 \sqrt{\pi^2 \overline\gamma ^2 A^2 k ^2 + \pi \tau k^3 M - \pi^2 k^2 \overline\gamma ^2 + k A g M^2} \over k^2 \pi \tau} \ep - {M \over \tau k \pi} \ep^2
$$
and
$$
\lambda_{l_k}(c_\pm(k)+\ep) =  \mp {2 \sqrt{\pi^2 \overline\gamma ^2 A^2 k ^2 + \pi \tau k^3 M - \pi^2 k^2 \overline\gamma ^2 + k A g M^2} \over A g M} \ep - {k \over A g} \ep^2.$$

Now assume that $k \in \mathcal{K}$. The second condition in the definition of $\mathcal{K}$ shows that $l_k \notin \N/\left\{k\right\}$  and thus $N_0(c_\pm(k)) = 1$. The first condition guarantees that
coefficient of $\ep$ in the expansion of $\lambda_k(c_\pm(k) + \ep)$ is non-zero. Thus we see $\lambda_\pm(c)$ changes sign as $c$ passes through $c_\pm(k)$: the crossing number is equal to one and thus is odd. Thus we have shown the ``if" direction in the corollary.
The ``only if" direction follows by showing that the crossing number is either two or zero when one of the conditions is not met. The details are simple so we omit them.

To show that $\left \vert \mathcal{K} \right \vert = \infty$, observe that, since  $\tau, M > 0$, we have
$$\ds
\lim_{k \to \infty} \left( \pi^2 \overline\gamma ^2 A^2 k ^2 + \pi \tau k^3 M - \pi^2 k^2 \overline\gamma ^2 + k A g M^2 \right) = \infty.
$$
Therefore the first condition in the definition of $\mathcal{K}$ is met for all $k$ sufficiently large. Likewise, no matter the choices of the parameters
$A,g,M$ and $\tau$, $\ds \lim_{k \to \infty} AgM/\pi \tau k =0 $ and the second
condition holds for $k$ sufficiently large. Thus $\left \vert \mathcal{K} \right \vert$ contains all $k > k_0$ for some finite $k_0 \in \N$.
 
\end{proof}

\subsection {Application of the abstract theorem}  
An appeal to Theorem \ref{glbif} has the following consequence.  
\begin{theorem} {\bf (Global bifurcation)}  \label{technical version}
Let the surface tension $\tau>0$, period $M>0$, 
Atwood number $A\in\R$, and average vortex strength $\overline\gamma\in\R$ be given.   
Let $\U$,
 $\mathcal{K}$ and $c_\pm(k)$ be defined as above.
Let $\mathcal{S}  \subset \U$ be the closure (in $\H^1_\per \times \H^1_{\per,0} \times\R$) of the set of all solutions of \eqref{TW v5}
for which either $\theta \ne 0$ or $\gamma_1 \ne 0$.
Given $k \in \mathcal{K}$,  
let $\mathcal{C}_\pm(k)$ be the connected component of $\mathcal{S}$ which contains $(0,0;c_\pm(k))$.

Then
\begin{enumerate}[(I)]
\item $either\ \mathcal{C}_\pm(k)$ is unbounded;

\item $or\ \mathcal{C}_\pm(k) = \mathcal{C}_+(l)$ or $\mathcal{C}_\pm(k) = \mathcal{C}_-(l)$
 for some $l \in \mathcal{K}$ with $l \ne k$; 
 
 \item $or\ \mathcal{C}_\pm(k)$ contains a point on the boundary of $\U$.
\end{enumerate}
\end{theorem}

\begin{proof}
We saw in Theorem \ref{TWE5} that the mapping $(\Theta,\Gamma)$ was compact on $\U_{b,h}$ for all $b,h>0$ and we saw that there are always choices of $c$ which result
in an odd crossing number in Corollary \ref{crossing}. Thus Theorem \ref{glbif} can be applied with outcomes which coincide with the outcomes $(I)$-$(III)$ 
in Theorem \ref{technical version}  except with the replacement of $\U$ with the $\U_{b,h}$ in $(III)$. Since $\U = \cup_{b,h>0} \U_{b,h}$ an easy topological argument 
gives $(III)$ as above.
\end{proof}

This general statement leads in turn to our main conclusion.  
\begin{proof} [Proof of Theorem \ref{main result}]
By Proposition \ref{TWE5}, a solution $(\theta,\gamma_1;c)$ of \eqref{TW v5} gives rise to 
symmetric periodic traveling wave solutions of the two dimensional gravity-capillary vortex sheet problem by
taking $z(\alpha,t)=ct + \widetilde{Z}[\theta](\alpha)$ and $\gamma = \overline\gamma  + \gamma_1.$ (Note that, as in the proof of Proposition \ref{TWE2}, 
this implies that $N=\widetilde{N}[\theta]$ and $T=\widetilde{T}[\theta].$ We will use the two equivalent notations interchangeably in what follows.)
Note that in the statement of Theorem \ref{main result} it is stated that the traveling wave solutions are smooth, whereas the solutions given in Theorem \ref{technical version}
are stated to merely belong to $\H^1_\per \times \H^1_{\per,0} $.   
However, the maps $\Theta$ and $\Gamma $ in \eqref{TW v5} are smoothing. 
Therefore a simple bootstrap argument shows that $\theta$ and $\gamma_1$ 
are in $\H^s_\per$ for any $s$ and thus in $C^\infty$.   The corresponding traveling waves
are likewise smooth; the details are routine and omitted.

Each of the outcomes (a)-(f) in Theorem \ref{main result} corresponds to one of 
the alternatives (I)-(III) of Theorem \ref{technical version}.   Fix $k \in \mathcal{K}$.  
It is straightforward to see that 
 alternative (II) in Theorem \ref{technical version} is interpreted as outcome (e) in Theorem \ref{main result}.

Now consider alternative (III).   
If $(\theta,\gamma_1;c)\in \mathcal{C}_\pm(k)$ is on the boundary of $\U$, 
then inspection of the definition of $\U$ shows that 
\begin{equation} \label{altIII}
\text{either }  \ \overline{\cos\theta} = 0\quad 
\text{or} \quad\widetilde{Z}[\theta] = \widetilde{Z}[\Theta] \notin \CA^3_0.
 \end{equation}  
The reconstruction of the curve $S (t)$ from $\theta$ via $\widetilde{Z}$ 
(recalling that $\overline{\sin\theta} =0$ for solutions) shows that
the length of $S (t)$ per period is given by
$$
L[\theta]:=\int_0^{2\pi} \left \vert \partial_\alpha \widetilde{Z}[\theta](a)\right \vert da 
 = {M \over \overline{\cos\theta}}.  
$$
In case  $\overline{\cos\theta} = 0$,  
the length of the curve reconstructed from $(\theta,\gamma_1;c)$ is formally infinite.  
Since $(\theta,\gamma_1;c)$ is in the closure of the set of nontrivial traveling wave solutions, 
the more precise statement is that 
there is sequence of solutions whose lengths diverge, which is outcome (a). 

Now suppose we have the other alternative in \eqref{altIII}, 
namely, that  $\widetilde{Z}[\theta] \notin \CA_0^3$.  
Since $h=0$ in this space, it means that 
$$
\inf_{\alpha,\alpha' \in [0,2\pi] } \zeta = 0, \quad \text{ where } \zeta(\alpha,\alpha') := 
\LV{ \widetilde{Z}[\theta](\alpha') - \widetilde{Z}[\theta](\alpha)  \over \alpha' - \alpha  }  \RV.
$$
Moreover, 
\be                 \label{not equal}  
\lim_{\alpha' \to \alpha} \zeta(\alpha,\alpha') 
= \LV \partial_\alpha \widetilde{Z}[\theta](\alpha) \RV  
= \frac M{2\pi\overline{cos\theta} }    
= {L[\theta]\over 2\pi} \ge 1.
\ee
But clearly 
${(\widetilde{Z}[\theta](\alpha') - \widetilde{Z}[\theta](\alpha) ) / (\alpha' - \alpha)  }$ 
is a continuous function of $(\alpha,\alpha')$ for  $\alpha \ne \alpha'$.  
Thus its infimum, 
which vanishes,  is attained at some pair of values $(\alpha_\star,\alpha'_\star)$
where $\alpha_\star \ne \alpha'_\star$.    
Hence $\zeta(\alpha_*,\alpha'_*)=0$, 
which in turn implies that 
$$
 \LV \widetilde{Z}[\theta](\alpha'_\star) - \widetilde{Z}[\theta](\alpha_\star)  \RV = 0. $$
This means that the curve reconstructed from $\theta$ intersects itself, outcome~(d).

Now consider alternative (I).   
Then $\mathcal{C}_\pm(k)$ contains a sequence of solutions $\left\{ (\theta_n,\gamma_{1,n};c_n)\right\}$
for which 
$$
\lim_{n \to \infty} \left( |c_n| + \|\theta_n\|_{\H^1_\per} + \| \gamma_{1,n} \|_{\H^1_\per} \right) = \infty, 
$$
so that at least one of the three terms on the left diverges. 

By combining (\ref{reconstruct condition}), (\ref{this is z}), and (\ref{win}), we see that $|\partial_{\alpha}\widetilde{Z}[\theta_{n}]|=\sigma_{n},$ and
the length of one period of the interface is thus proportional to $\sigma_{n}.$
If $\sigma_{n}\rightarrow\infty,$ then outcome (a) has occurred; thus, we may assume that
$\sigma_{n}$ remains bounded above independently of $n.$  Since the length of one period of the curve may not vanish (by periodicity),
we also see that $\sigma_n$ is bounded below (away from zero) independently of $n.$
Considering again (\ref{reconstruct condition}), we see that $\sigma_n$ being bounded above implies that $\overline{\cos\theta_n}$
is bounded away from zero.

Suppose first that $|c_n|\to\infty$ 
but $\|\theta_n\|_{\H^1_\per} + \| \gamma_{1,n} \|_{\H^1_\per}$ is bounded. 
We see from \eqref{wn eqn} that 
$$\|c_{n}\sin\theta_{n}\|_{H^{1}}=\|\Re(W^{*}_{n}N_{n})\|_{H^{1}}.$$
Since $W_{n}^{*}=B[\widetilde{Z}[\theta_{n}]]\gamma_{n},$
this implies
$$\|c_{n}\sin\theta_{n}\|_{H^{1}}=\|\Re(B[\widetilde{Z}[\theta_{n}]]\gamma_{n}\ N_{n})\|_{H^{1}}.$$
We then use (\ref{WN is not smooth}) to write this as
$$\|c_{n}\sin\theta_{n}\|_{H^{1}}=\left\|\frac{1}{2\sigma_n}H\gamma_n + \Re(K[\widetilde{Z}[\theta_n]]\gamma_n\ N_n)\right\|_{H^{1}}.$$
We have remarked above that $\frac{1}{\sigma_n}$ is bounded independently of $n,$ and we see that $H\gamma_n$ is uniformly bounded in $H^{1}$
since $\gamma_{n}$ is.
Applying Lemma \ref{hilbert plus smooth} gives an estimate for the operator $K,$ and we find the following:
$$\|c_{n}\sin\theta_n\|_{H^{1}}\leq C\|\gamma_{n}\|_{H^{1}}+C\exp\{C\|\widetilde{Z}[\theta_n]\|_{H^{1}}\}\|\gamma_{n}\|_{H^{1}}\|N_{n}\|_{H^{1}}.$$
Since $\theta_{n}$ is, by assumption, bounded in $H^{1},$ we see from (\ref{this is z}) and (\ref{tangent def 2}) that $\widetilde{Z}[\theta_n]$ and $N_n$ are as well.
We conclude that
$c_n \sin\theta_n$ is bounded in $\H^{1}_{\per},$ independently of $n.$

Therefore $\|\sin\theta_n\|_{\H^1_\per} \to 0,$ and by Sobolev embedding,
$\sin\theta_n\to 0$ uniformly.
This implies that $\theta_n$ converges to a multiple of $\pi;$ 
the uniform convergence and the continuity and oddness of $\theta_n$ make it is straightforward to see that this multiple must be zero. 
Note also, then, that $|\cos \theta_n| \to 1,$ uniformly.    
Continuing, we integrate \eqref{gamma**} once, finding that the quantity
\be \label{abc}  
(c_n\cos\theta_n - \Re( W^*_n T_n))\gamma_n 
-2A\left(\frac{1}{8}\frac{\gamma_n^{2}}{\sigma_{n}}+
g\sigma_{n} \int^\alpha \sin\theta_n\, d\alpha  +{\sigma_{n} \over 2} \{c_n \cos\theta_n -\Re(W^*_n T_n)\}^2  \right)  \ee 
is bounded in $L^2_{per}$.  

Recalling again that $W_{n}^{*}=B[\widetilde{Z}(\theta_{n})]\gamma_{n},$ we see that  
(\ref{WT is smooth}) gives a formula for $\Re(W_{n}^*T_{n}).$  
Similarly to our previous use of Lemma \ref{hilbert plus smooth}, we see that 
Lemma \ref{hilbert plus smooth} then implies that
$\Re(W_{n}^{*}T_{n})$ is bounded in $\H^{1}_{\per}.$  
If $A=0$, we then deduce that $c_n\gamma_n\cos\theta_n  $ is bounded in $L^2_{per}$ 
and therefore $\gamma_n \to0$   in $L^2_{per}$.  
Thus the average, which is a constant, must satisfy $\overline{\gamma} = 0$.   
If we have $\bar{\gamma}\neq 0,$ then this is a contradiction.  

Now assume that $A\ne 0$.  Dividing (\ref{abc}) by $c_{n}^{2},$ we see that all 
the terms then go to zero as $n\to\infty$ except 
$\frac{\sigma_{n}}{2}(\cos\theta_n)^2.$  This then implies that 
$\frac{\sigma_{n}}{2}(\cos\theta_n)^2$ goes to zero, which is a contradiction 
since $\cos\theta_n\to 1$ uniformly.  

If $\bar{\gamma}=0$ and $A=0,$ then we do not rule out $|c_{n}|\rightarrow\infty;$ this is
possibility (f) of the theorem.

If $\| \theta_n\|_{\H^1_\per}$ diverges, then either  $\theta_n$ or $\partial_\alpha \theta_n$ 
diverges in $L^2_\per$.
Since $\theta$ is the tangent angle to $S (t)$, the curvature is exactly 
$\kappa(\alpha) = \partial_\alpha \theta(\alpha)/\sigma.  $
Inspection of (\ref{reconstruct condition}) indicates that if $\sigma_{n}\rightarrow0,$ then $\overline{\cos\theta_{n}}\rightarrow\infty.$  This clearly cannot be the case, however, and thus $\sigma_{n}$
cannot go to zero.   
Recall that we assume that $\sigma_{n}$ is bounded above.
If it is $\partial_\alpha \theta_n$ that diverges, then 
we see that the curvature diverges, which is to say that we have outcome (b).
On the other hand, suppose that it is $\theta_n$ that diverges in $L^2_\per$.  
Since $\theta_n$ is odd and periodic, $\theta_n(0)=\theta_n(2 \pi) = 0$.
If the $L^2_\per$-norm of $\theta_n$ diverges then of course its $L^\infty$-norm also diverges 
and so does $\partial_\alpha \theta_n$. 
Which means that the curvature for the reconstructed interface is diverging. 
Thus  $\|\theta_n\|_{\H^1_\per}$ diverging implies~ outcome~(b).

If  $\| \gamma_{1,n}\|_{\H^1_\per}$ diverges, then either  
$\gamma_{1,n}$ or $\partial_\alpha \gamma_{1,n}$ diverges in $L^2_\per$. 
Suppose that it is the former. 
From Section \ref{eom}, the jump in the tangential velocity across the interface 
is related to $\gamma$ by \eqref{jump}.
By the reconstruction method and the length $L$ above, it implies that 
$ j(\alpha,t) =2 \pi (\overline\gamma  + \gamma_1(\alpha,t))/ L[\theta].$
If $L[\theta_n]$ remains bounded, then clearly the jump $j$ given above diverges in $L^2_\per$, 
which is outcome (c).       If $L[\theta_n]$ diverges, we have outcome (a).
Likewise, if $\partial_\alpha \gamma_{1,n}$ diverges in $L^2_\per$, then 
either the derivative of the jump diverges, outcome (c), or else the length diverges, outcome (a).
\end{proof}

\bibliographystyle{plain}
\bibliography{globalbif}{}

\begin{thebibliography}{10}

\bibitem{ablowitzFokas}
M.J. Ablowitz and A.S. Fokas.
\newblock {\em Complex variables: introduction and applications}.
\newblock Cambridge Texts in Applied Mathematics. Cambridge University Press,
  Cambridge, 1997.

\bibitem{AAW}
B.~Akers, D.M. Ambrose, and J.D. Wright.
\newblock Traveling waves from the arclength parameterization: vortex sheets
  with surface tension.
\newblock {\em Interfaces Free Bound.}, 15(3):359--380, 2013.

\bibitem{aapwPreprint}
B.F. Akers, D.M. Ambrose, K.~Pond, and J.D. Wright.
\newblock Internal capillary gravity waves.
\newblock 2014.
\newblock In preparation.

\bibitem{AAW2}
B.F. Akers, D.M. Ambrose, and J.D. Wright.
\newblock Gravity perturbed {C}rapper waves.
\newblock {\em Proc. R. Soc. Lond. Ser. A Math. Phys. Eng. Sci.},
  470(2161):20130526, 14, 2014.

\bibitem{A}
D.M. Ambrose.
\newblock Well-posedness of vortex sheets with surface tension.
\newblock {\em SIAM J. Math. Anal.}, 35(1):211--244 (electronic), 2003.

\bibitem{ambroseMasmoudi2}
D.M. Ambrose and N.~Masmoudi.
\newblock Well-posedness of 3{D} vortex sheets with surface tension.
\newblock {\em Commun. Math. Sci.}, 5(2):391--430, 2007.

\bibitem{amickTurner1}
C.J. Amick and R.E.L. Turner.
\newblock A global theory of internal solitary waves in two-fluid systems.
\newblock {\em Trans. Amer. Math. Soc.}, 298(2):431--484, 1986.

\bibitem{amickTurner2}
C.J. Amick and R.E.L. Turner.
\newblock Small internal waves in two-fluid systems.
\newblock {\em Arch. Rational Mech. Anal.}, 108(2):111--139, 1989.

\bibitem{BHL}
J.T. Beale, T.Y. Hou, and J.S. Lowengrub.
\newblock Growth rates for the linearized motion of fluid interfaces away from
  equilibrium.
\newblock {\em Comm. Pure Appl. Math.}, 46(9):1269--1301, 1993.

\bibitem{splash1}
A.~Castro, D.~Cordoba, C.~Fefferman, F.~Gancedo, and J.~Gomez-Serrano.
\newblock Finite time singularities for water waves with surface tension.
\newblock {\em J. Math. Phys.}, 53(11):--, 2012.

\bibitem{constantinStrauss}
A.~Constantin and W.~Strauss.
\newblock Exact steady periodic water waves with vorticity.
\newblock {\em Comm. Pure Appl. Math.}, 57(4):481--527, 2004.

\bibitem{walterPreprint}
A.~Constantin, W.~Strauss, and E.~Varvaruca.
\newblock Global bifurcation of steady gravity water waves with critical
  layers.
\newblock Preprint.

\bibitem{splash3}
D.~Coutand and S.~Shkoller.
\newblock On the impossibility of finite-time splash singularities for vortex
  sheets.
\newblock 2014.
\newblock Preprint. arXiv:1407.1479.

\bibitem{Crapper}
G.D. Crapper.
\newblock An exact solution for progressive capillary waves of arbitrary
  amplitude.
\newblock {\em J. Fluid Mech.}, 2:532--540, 1957.

\bibitem{deBoeck2}
P.~de~Boeck.
\newblock Existence of capillary-gravity waves that are perturbations of
  {C}rapper's waves.
\newblock Preprint.

\bibitem{deBoeck}
P.~de~Boeck.
\newblock Global bifurcation for steady finite-depth capillary-gravity waves
  with constant vorticity.
\newblock Preprint.

\bibitem{splash2}
C.~Fefferman, A.D. Ionescu, and V.~Lie.
\newblock On the absence of {``splash''} singularities in the case of two-fluid
  interfaces.
\newblock 2013.
\newblock Preprint. arXiv.1312.2917.

\bibitem{globalCapillary}
P.~Germain, N.~Masmoudi, and J.~Shatah.
\newblock Global existence for capillary water waves.
\newblock {\em Comm. Pure Appl. Math.}, pages n/a--n/a, 2014.

\bibitem{HLS1}
T.Y. Hou, J.S. Lowengrub, and M.J. Shelley.
\newblock Removing the stiffness from interfacial flows with surface tension.
\newblock {\em J. Comput. Phys.}, 114(2):312--338, 1994.

\bibitem{HLS2}
T.Y. Hou, J.S. Lowengrub, and M.J. Shelley.
\newblock The long-time motion of vortex sheets with surface tension.
\newblock {\em Phys. Fluids}, 9(7):1933--1954, 1997.

\bibitem{hur-nonlinearity}
V.M. Hur.
\newblock No solitary waves exist on 2{D} deep water.
\newblock {\em Nonlinearity}, 25(12):3301--3312, 2012.

\bibitem{K}
H.~Kielh{\"o}fer.
\newblock {\em Bifurcation theory}, volume 156 of {\em Applied Mathematical
  Sciences}.
\newblock Springer, New York, second edition, 2012.
\newblock An introduction with applications to partial differential equations.

\bibitem{kinnersley}
W.~Kinnersley.
\newblock Exact large amplitude capillary waves on sheets of fluid.
\newblock {\em J. Fluid Mech.}, 77:229--241, 8 1976.

\bibitem{matioc}
B.-V. Matioc.
\newblock Global bifurcation for water waves with capillary effects and
  constant vorticity.
\newblock {\em Monatsh. Math.}, 174(3):459--475, 2014.

\bibitem{R}
P.H. Rabinowitz.
\newblock Some global results for nonlinear eigenvalue problems.
\newblock {\em J. Functional Analysis}, 7:487--513, 1971.

\bibitem{toland-pseudo}
J.F. Toland.
\newblock On a pseudo-differential equation for {S}tokes waves.
\newblock {\em Arch. Ration. Mech. Anal.}, 162(2):179--189, 2002.

\bibitem{walsh}
S.~Walsh.
\newblock Steady stratified periodic gravity waves with surface tension {II}:
  global bifurcation.
\newblock {\em Discrete Contin. Dyn. Syst.}, 34(8):3287--3315, 2014.

\end{thebibliography}

\end{document}